\documentclass[oneside,11pt]{amsart}
\usepackage{amsmath, amsfonts,amsthm,times,graphics}
\usepackage{mathrsfs}
\usepackage[usenames]{color}
\usepackage[active]{srcltx}
 \makeatletter
\renewcommand*\subjclass[2][2000]{%
  \def\@subjclass{#2}%
  \@ifundefined{subjclassname@#1}{%
    \ClassWarning{\@classname}{Unknown edition (#1) of Mathematics
      Subject Classification; using '1991'.}%
  }{%
    \@xp\let\@xp\subjclassname\csname subjclassname@#1\endcsname
  }%
}
 \makeatother

\newtheorem*{ThmA}{Theorem A}
\newtheorem*{ThmB}{Theorem B}
\newtheorem*{ThmC}{Theorem C}
\newtheorem*{ThmD}{Theorem D}
\newtheorem*{ThmE}{Theorem E}
\newtheorem*{ThmF}{Theorem F}
\newtheorem*{ThmG}{Theorem G}
\newtheorem*{ThmH}{Theorem H}
\newtheorem*{ThmI}{Theorem I}

\newtheorem*{ThmK}{Theorem K}
\newtheorem*{ThmL}{Theorem L}
\newtheorem*{ThmM}{Theorem M}
\newtheorem*{ThmN}{Theorem N}
\newtheorem*{ThmO}{Theorem O}

\newtheorem*{LemJ}{Lemma J}

\newtheorem{Thm}{Theorem}[section]
\newtheorem{Cor}[Thm]{Corollary}
\newtheorem{Lem}[Thm]{Lemma}
\newtheorem{Pro}[Thm]{Proposition}
\theoremstyle{definition}
\newtheorem{Def}[Thm]{Definition}

\theoremstyle{remark}
\newtheorem{Rem}[Thm]{\upshape\bfseries Remark}

\numberwithin{equation}{section}

\newcommand{\ee}{\mathrm{e}}

\theoremstyle{definition}

\def\be{\begin{equation}}
\def\ee{\end{equation}}

\newcommand{\ben}{\begin{enumerate}}
\newcommand{\een}{\end{enumerate}}

\newcommand{\br}{\begin{rem}}
\newcommand{\er}{\end{rem}}
\newcommand{\brs}{\begin{rems}}
\newcommand{\ers}{\end{rems}}
\newcommand{\bo}{\begin{obser}}
\newcommand{\eo}{\end{obser}}
\newcommand{\bos}{\begin{obsers}}
\newcommand{\eos}{\end{obsers}}
\newcommand{\bpf}{\begin{pf}}
\newcommand{\epf}{\end{pf}}
\newcommand{\ba}{\begin{array}}
\newcommand{\ea}{\end{array}}
\newcommand{\beq}{\begin{eqnarray}}
\newcommand{\beqq}{\begin{eqnarray*}}
\newcommand{\eeq}{\end{eqnarray}}
\newcommand{\eeqq}{\end{eqnarray*}}

\numberwithin{equation}{section}

\newcounter{minutes}
\divide\time by 60
\newcounter{hours}
\multiply\time by 60 \addtocounter{minutes}{-\time}
\begin{document}
\title{On Riesz type inequalities, Hardy-Littlewood type theorems and smooth moduli}

\author[Shaolin Chen and Hidetaka Hamada]{Shaolin Chen and Hidetaka Hamada}

\address{S. L.  Chen, College of Mathematics and
Statistics, Hengyang Normal University, Hengyang, Hunan 421002,
People's Republic of China; Hunan Provincial Key Laboratory of Intelligent Information Processing and Application,  421002,
People's Republic of China.} \email{mathechen@126.com}

\address{H. Hamada, Faculty of Science and Engineering, Kyushu Sangyo University,
3-1 Matsukadai 2-Chome, Higashi-ku, Fukuoka 813-8503, Japan.}
\email{ h.hamada@ip.kyusan-u.ac.jp}





\maketitle

\def\thefootnote{}
\footnotetext{2010 Mathematics Subject Classification. Primary
42A50, 42B30, 46E15; Secondary 31B05, 31C10, 32U05.}
\footnotetext{Keywords. Riesz type inequality,  Hardy-Littlewood
type Theorem, Smooth moduli}
\makeatletter\def\thefootnote{\@arabic\c@footnote}\makeatother

\begin{abstract}
The  purpose of this paper is to develop some methods  to study
Riesz type inequalities,  Hardy-Littlewood type theorems and smooth
moduli of holomorphic, pluriharmonic  and harmonic functions in
high-dimensional cases. Initially, we  prove some sharp  Riesz type
inequalities of pluriharmonic functions on bounded symmetric
domains. The obtained results extend the main results  in
(\textit{Trans. Amer. Math. Soc.} {\bf 372} (2019)~ 4031--4051).
Furthermore, some Hardy-Littlewood type theorems of holomorphic and
pluriharmonic functions on  John domains are  established.
Additionally, we also discuss the Hardy-Littlewood type theorems and
smooth moduli of holomorphic,
 pluriharmonic  and harmonic functions. Consequently, we improve and generalize the corresponding results
in (\textit{Acta Math.} {\bf 178} (1997)~ 143--167) and
(\textit{Adv. Math.} {\bf 187} (2004)~ 146--172).
\end{abstract}

\maketitle \pagestyle{myheadings} \markboth{ S. L. Chen and H. Hamada }{On Riesz type inequalities, Hardy-Littlewood type theorems and smooth moduli}

\tableofcontents

\section{Introduction}\label{Intro}
Denote by $\mathbb{C}^{n}$  the  complex space of dimension $n$,
where $n$ is a positive integer. We can also interpret
$\mathbb{C}^{n}$ as the real $2n$-space $\mathbb{R}^{2n}$. For
$z=(z_{1},\ldots,z_{n}),~w=(w_{1},\ldots,w_{n})\in\mathbb{C}^{n} $,
the standard
 Hermitian scalar product on $\mathbb{C}^{n}$ and the
Euclidean norm of $z$ are given by $$\langle z,w\rangle :=
\sum_{k=1}^nz_k\overline{w}_k~\mbox{and}~ |z|:={\langle
z,z\rangle}^{1/2}, $$ respectively. For $a\in \mathbb{C}^n$ ($a\in
\mathbb{R}^m$ resp.), we use $\mathbb{B}^n(a,r)=\{z\in
\mathbb{C}^{n}:\, |z-a|<r\}$ ($\mathbf{B}^m(a,r)=\{x\in
\mathbb{R}^{m}:\, |x-a|<r\}$ resp.) to denote the (open) ball of
radius $r>0$ with center $a$, where $m\geq2$ is a positive integer.
Let $\mathbf{S}^{m-1}(a,r):=\partial\mathbf{B}^{m}(a,r)$ be the
boundary of $\mathbf{B}^{m}(a,r)$. Then the boundary of
$\mathbb{B}^{n}(a,r)$ is
$\partial\mathbb{B}^{n}(a,r)=\mathbf{S}^{2n-1}(a,r)$. In particular,
let $\mathbb{B}^{n}:=\mathbb{B}^{n}(0,1)$
($\mathbf{B}^{m}:=\mathbf{B}^{m}(0,1)$ resp.)  and
$\mathbf{S}^{m-1}:=\mathbf{S}^{m-1}(0,1)$.

A two times continuously differentiable real-valued function $u$
defined in an open subset $\mathscr{M}$ of $\mathbb{R}^{m}$ is said
to be harmonic if $\Delta\,u(x)=0$ for all
$x=(x_{1},\ldots,x_{m})\in\mathscr{M}$, where
$$\Delta:=\sum_{j=1}^{m}\frac{\partial^{2}}{\partial\,x^{2}_{j}}$$ is a Laplace operator (see \cite{ABR}). In particular,
a two times continuously differentiable complex-valued function $f$
defined in an open subset $\mathscr{D}$ of $\mathbb{C}^{n}$ is said
to be  pluriharmonic if it satisfies the $n^{2}$ partial
differential equations

$$\frac{\partial^{2}f}{\partial z_{j}\partial\overline{z}_{k}}=0$$ for $j,k\in\{1,\ldots,n\}$ (see \cite{DHK-2011,RU,Ru-1,Vl}).
In \cite{Ru}, Rudin showed that a two times continuously
differentiable complex-valued function $f$ defined in
$\mathbb{B}^{n}$
 is pluriharmonic if and only if $f$ satisfies
$$\Delta f=0~\mbox{and}~\tilde{\Delta}f=0,$$
where
$$\tilde{\Delta}=(1-|z|^{2})\left(\Delta f-4\sum_{j,k=1}^{n}z_{j}\overline{z}_{k}\frac{\partial^{2}f}{\partial z_{j}\partial\overline{z}_{k}}\right)$$
is the  Laplace-Beltrami operator. In a simply connected domain
$\mathscr{D}\subset\mathbb{C}^{n}$, a pluriharmonic function
$f:\,\mathscr{D}\rightarrow \mathbb{C}:=\mathbb{C}^{1}$  has a
representation $f=h+\overline{g},$ where $h$ and $g$ are holomorphic
in $\mathscr{D}$ (see \cite{CH-2022,CHPV,DHK-2011,Vl}); this
representation is unique up to an additive constant. Furthermore,
from this representation, it is easy to know that pluriharmonic
 functions are a broader class of functions than holomorphic functions.
In particular, if $n=1$, then the pluriharmonic function  is
equivalent to the harmonic function.

Throughout of this paper,
 we use the symbol $C$ to denote the various positive
constants, whose value may change from one occurrence to another.

\section{Preliminaries and main results}\label{sec2}
\subsection{The Riesz type inequalities and the Hardy-Littlewood type Theorems}

Let $\Omega\subset\mathbb{C}^{n}$ be a bounded symmetric domain with
origin and Bergman-Silov boundary $b$. Denote by $\Gamma$ the group
of holomorphic automorphisms of $\Omega$, and denote by $\Gamma_{0}$
the isotropy group of $\Gamma$. It is well known that $\Omega$ is
circular and star-shaped with respect to $0$ and that $b$ is
circular. The group $\Gamma_{0}$ is transitive on $b$ and $b$ has a
unique normalized
 $\Gamma_{0}$-invariant measure $\sigma$ with $\sigma(b)=1$ (see \cite{Bo,H-M,Hua,KW,M,Shi, WK}). Obviously, the unit polydisk and the
 unit ball $\mathbb{B}^{n}$ are bounded symmetric domains.
For $p\in(0,\infty]$, the pluriharmonic Hardy space
$\mathscr{PH}^{p}(\Omega)$ consists of all those pluriharmonic
functions $f:~\Omega\rightarrow\mathbb{C}$ such that, for
$p\in(0,\infty)$,
$$\|f\|_{p}:=\sup_{r\in[0,1)}M_{p}(r,f)<\infty,$$ and, for $p=\infty$, $$\|f\|_{\infty}:=\sup_{r\in[0,1)}M_{\infty}(r,f)<\infty,$$
where
$$M_{p}(r,f)=\left(\int_{b}|f(r\zeta)|^{p}d\sigma(\zeta)\right)^{\frac{1}{p}}~\mbox{and}~M_{\infty}(r,f)=\sup_{\zeta\in\,b}|f(r\zeta)|.$$

Suppose that $\omega\in L_{loc}^{1}(\mathbb{R}^{m})$ with $\omega>0$
a.e..
 For $p\in(0,\infty)$, we denote by $L^{p}(E,\omega)$ the space of all
 measurable functions $f$ with

$$\|f\|_{p,E,\omega}=\left(\int_{E}|f(x)|^{p}\omega(x)d\mu(x)\right)^{\frac{1}{p}}<\infty,$$ where
$E\subset\mathbb{R}^{m}$ is a $m$-dimensional Lebesgue measurable
set and
 $d\mu$ is the Lebesgue measure  on $\mathbb{R}^{m}$.
 For the $m$-dimensional Lebesgue measurable set $E$ with
\[
0<\int_{E}\omega(x)d\mu(x)<\infty,
\]let $$f_{E,\omega}=\frac{\int_{E}f(x)\omega(x)d\mu(x)}{\int_{E}\omega(x)d\mu(x)}$$ be the average,
 where $f$ is a measurable function  in $E$.
In particular, we write $\|f\|_{p,E}:=\|f\|_{p,E,\omega}$ and
$f_{E}:=f_{E,\omega}$ when $\omega\equiv1$. We say that $\omega$ is
a Muckenhoupt weight, $\omega\in A_{M}^{q}(\Omega)$, where
$q\in(1,\infty)$ and $M\in[1,\infty)$, if $\omega>0$ a.e. on
$\Omega$, and
$$\|\omega\|_{1,Q}\leq M|Q|^{q}\|\omega\|_{1/(1-q),Q}$$
for each cube or ball $Q$ contained in the domain
$\Omega\subset\mathbb{R}^{m}$.

Let's recall one of the celebrated results on holomorphic functions
in $\mathscr{PH}^{p}(\mathbb{D})$  by Riesz, where
$\mathbb{D}:=\mathbb{B}^{1}$ and $p\in(1,\infty)$.

\begin{ThmA}{\rm (M. Riesz)}\label{M-R}
If a real-valued harmonic function
$u\in\mathscr{PH}^{p}(\mathbb{D})$ for some $p\in(1,\infty)$, then
its harmonic conjugate $v$ with $v(0)=0$ is also of class
$\mathscr{PH}^{p}(\mathbb{D})$. Furthermore, there is a constant
$C$, depending only on $p$, such that

\begin{equation}\label{Riez}
M_{p}(r,v)\leq
CM_{p}(r,u)
\end{equation}

for all $r\in[0,1)$.
\end{ThmA}

By replacing $\|\cdot\|_{p}$ with $\|\cdot\|_{p,\mathbb{D}}$ in
Theorem A, Hardy and Littlewood  \cite{H-L} proved the following
result for $p\in(0,\infty)$.

\begin{ThmB}\label{ThmA}
For each $p\in(0,\infty)$, there exists a constant $C$, depending
only on $p$, such that
$$\|u-u(0)\|_{p,\mathbb{D}}\leq C\|v\|_{p,\mathbb{D}}$$
for all real valued harmonic functions $u$ and $v$ in the unit disk
$\mathbb{D}$ such that $u+iv$ is holomorphic  in $\mathbb{D}$, where
$d\mu$ is the Lebesgue area measure on $\mathbb{D}$.
\end{ThmB}

In \cite{Pi}, Pichorides  improved (\ref{Riez}) into the following
sharp form.
\begin{equation}\label{Riez-1}
\|v\|_{p}\leq\cot\frac{\pi}{2p^*}\|u\|_{p},
\end{equation} where $p^*=\max\{p,~p/(p-1)\}$. Later, Verbitsky
\cite{Ve} further improved  (\ref{Riez-1}) and obtained the
following sharp inequalities
\begin{equation}\label{Riez-2}\frac{1}{\cos\frac{\pi}{2p^*}}\|v\|_{p}\leq\|f\|_{p}\leq\frac{1}{\sin\frac{\pi}{2p^*}}\|u\|_{p},
\end{equation}
where $f=u+iv$ is holomorphic.  As an analogy of Theorem A, Kalaj
\cite{K-2019} established the following sharp Riesz type
inequalities for complex-valued harmonic functions in
$\mathscr{PH}^{p}(\mathbb{D})$ with $p\in(1,\infty)$.  

\begin{ThmC}{\rm (\cite[Theorems 2.1 and 2.3]{K-2019})}\label{Trm-2019}
Let $p\in(1,\infty)$ be a constant and
$f=h+\overline{g}\in\mathscr{PH}^{p}(\mathbb{D})$, where $h$ and $g$
are holomorphic functions in $\mathbb{D}$.

\begin{enumerate}
\item[{\rm ($\mathscr{A}_{1}$)}]  If $ {\rm Re}(g(0)h(0))\geq0$, then,
$$\left(\int_{0}^{2\pi}\big(|h(e^{i\theta})|^{2}+
|g(e^{i\theta})|^{2}\big)^{\frac{p}{2}}\frac{d\theta}{2\pi}\right)^{\frac{1}{p}}\leq
\frac{1}{C_{1}(p)} \|f\|_{p},$$ where
$C_{1}(p)=\sqrt{1-\big|\cos\frac{\pi}{p}\big|}$ and  $``{\rm Re}"$
denotes the real part of a complex number.
\item[{\rm ($\mathscr{A}_{2}$)}] If ${\rm Re}(g(0)h(0))\leq0$, then,
$$\|f\|_p\leq
C_{2}(p)\left(\int_{0}^{2\pi}(|h(e^{i\theta})|^{2}+
|g(e^{i\theta})|^{2})^{\frac{p}{2}}\frac{d\theta}{2\pi}\right)^{\frac{1}{p}},$$
where
$C_{2}(p)=\sqrt{2}\max\left\{\sin\frac{\pi}{2p},~\cos\frac{\pi}{2p}\right\}$.
\end{enumerate}
\end{ThmC}

Whether there is a Riesz type theorem in high-dimensional space
 has always been a challenging problem (see \cite[p.167-172]{FS-1972}).
Let $u=(u_{0},u_{1},\ldots,u_{n})$ be a vector valued harmonic
function
 in $(n+1)$-dimensional upper half space
$\mathbb{R}_{+}^{n+1}$ satisfying the Cauchy-Riemann system

$$\sum_{j=0}^{n}\frac{\partial\,u_{j}}{\partial\,x_{j}}=0,~
\frac{\partial\,u_{k}}{\partial\,x_{j}}=\frac{\partial\,u_{j}}{\partial\,x_{k}},~0\leq\,k,j\leq\,n.$$
 By
analogy with (\ref{Riez-2}), Fefferman and Stein showed that
$u\in\,L^{p}$ if and only if $u_{0}\in\,L^{p}$. They also used the
non-tangential maximal function to establish a Riesz type Theorem in
$\mathbb{R}_{+}^{n+1}$ (see \cite[Theorem 9]{FS-1972} and
\cite{FS-2020}). The first aim of this paper is to establish some
Riesz type inequalities  of pluriharmonic functions on bounded
symmetric domains, which extend Theorem C.

\begin{Thm}\label{thm-1}
Let $\Omega\subset\mathbb{C}^{n}$ be a bounded symmetric domain with
origin and Bergman-Silov boundary $b$, and let $p\in[1,\infty]$ be a
given constant. Let $f=h+\overline{g}$, where $h$ and $g$ are
holomorphic functions in $\Omega$. Then the following statements
hold.

\begin{enumerate}
\item[{\rm ($\mathscr{B}_{1}$)}] If $p\in(1,\infty)$, then $f=h+\overline{g}\in\mathscr{PH}^{p}(\Omega)$ if and only if $h,~g\in\mathscr{PH}^{p}(\Omega)$;
Moreover, \begin{enumerate}
\item[{\rm ($\mathscr{B}_{1}{\rm(I)}$)}]  If $f=h+\overline{g}\in\mathscr{PH}^{p}(\Omega)$ and $ {\rm Re}(g(0)h(0))\geq0$, then,

$$\lim_{r\rightarrow1^{-}}\left(\int_{b}\big(|h(r\zeta)|^{2}+
|g(r\zeta)|^{2}\big)^{\frac{p}{2}}d\sigma(\zeta)\right)^{\frac{1}{p}}\leq
\frac{1}{C_{1}(p)} \|f\|_{p},$$ where $C_{1}(p)$ is the same as in
Theorem C.
\item[{\rm ($\mathscr{B}_{1}{\rm(II)}$)}] If $f=h+\overline{g}\in\mathscr{PH}^{p}(\Omega)$ and ${\rm Re}(g(0)h(0))\leq0$, then,

$$\|f\|_p\leq
C_{2}(p)\lim_{r\rightarrow1^{-}}\left(\int_{b}\big(|h(r\zeta)|^{2}+
|g(r\zeta)|^{2}\big)^{\frac{p}{2}}d\sigma(\zeta)\right)^{\frac{1}{p}},$$
where $C_{2}(p)$ is the same as in Theorem C.
\end{enumerate}
\item[{\rm ($\mathscr{B}_{2}$)}] If $p=1$, then $f=h+\overline{g}\in\mathscr{PH}^{1}(\Omega)$ implies that $h,~g\in\mathscr{PH}^{q}(\Omega)$ for all $q\in(0,1)$.
Furthermore, if $f=h+\overline{g}\in\mathscr{PH}^{1}(\Omega)$, then
\begin{equation}\label{Ch-Ha-1}
M_{1}(r,h)=O\left(\log\frac{1}{1-r}\right)~\mbox{and}~M_{1}(r,g)=O\left(\log\frac{1}{1-r}\right)
\end{equation} as $r\rightarrow1^{-}$. In particular, if $\Omega$ is
the unit polydisk, then the estimates of  {\rm (\ref{Ch-Ha-1})} are
sharp.
\item[{\rm ($\mathscr{B}_{3}$)}] If $p=\infty$, then  there are two unbounded holomorphic functions $h$ and $g$ such that $f=h+\overline{g}\in\mathscr{PH}^{\infty}(\Omega)$.
\end{enumerate}
\end{Thm}

As in \cite[Corollaries 2.2 and 2.5]{K-2019}, we obtain the
following corollary from Theorem \ref{thm-1}, which generalizes the
Riesz type inequalities (\ref{Riez}), (\ref{Riez-1}) and
(\ref{Riez-2}) to real valued pluriharmonic functions on bounded
symmetric domains in $\mathbb{C}^{n}$.

\begin{Cor}\label{Sharp-1}
Let $\Omega\subset\mathbb{C}^{n}$ be a bounded symmetric domain with
origin and Bergman-Silov boundary $b$, and let $1<p<\infty$. If $u$
and $v$ are real pluriharmonic functions on $\Omega$ with $v(0)=0$
and $g=u+iv$ is a holomorphic function on $\Omega$, then, we have
\begin{equation}\label{Sharp-2}
 \| v\|_p\leq \cos\frac{\pi}{2p^*} \| g\|_p
\end{equation}
and
\begin{equation}\label{Sharp-3} \| g\|_p\leq
\frac{1}{\sin\frac{\pi}{2p^*}} \| u\|_p.
\end{equation}
\end{Cor}

\begin{Rem}
If $n=1$ and $\Omega$ is the unit disk in Corollary \ref{Sharp-1},
then the inequalities (\ref{Sharp-2}) and (\ref{Sharp-3}) are sharp
(see \cite{K-2019,Ve} and \cite[Theorem 3.7]{Pi}).
\end{Rem}

Let $\mathbf{b}^p(\mathbb{B}^n)$ denote the Bergman class of
pluriharmonic functions $f$ defined on $\mathbb{B}^n$, satisfying
the condition
\[
\| f\|_{\mathbf{b}^p(\mathbb{B}^n)}:=\left( \int_{\mathbb{B}^n}
|f(z)|^p{d\mu(z)}\right)^{1/p}<\infty.
\]
As in \cite[Corollary 2.9]{K-2019}, we also obtain the following
corollary from Theorem \ref{thm-1}.

\begin{Cor}
Let $1<p<\infty$. If $h$ and $g$ are holomorphic functions on
$\mathbb{B}^n$, $f=h+\overline{g}\in{\mathbf{b}^p(\mathbb{B}^n)}$
and $ {\rm Re}(g(0)h(0))=0$, then we have
\[
\left(\int_{\mathbb{B}^n}\big(|h(z)|^{2}+
|g(z)|^{2}\big)^{\frac{p}{2}}{d\mu(z)}\right)^{\frac{1}{p}}\leq
\frac{1}{C_{1}(p)}
\left(\int_{\mathbb{B}^n}|f(z)|^{p}{d\mu(z)}\right)^{\frac{1}{p}},
\]
and
\[
\left(\int_{\mathbb{B}^n}|f(z)|^{p}{d\mu(z)}\right)^{\frac{1}{p}}
\leq C_{2}(p)\left(\int_{\mathbb{B}^n}\big(|h(z)|^{2}+
|g(z)|^{2}\big)^{\frac{p}{2}}{d\mu(z)}\right)^{\frac{1}{p}},
\]
where $C_{1}(p)$ and $C_{2}(p)$ are the same as in Theorem C.
\end{Cor}

\begin{Def}\label{eq-de-1}
Let $\Omega\subset\mathbb{C}^n$ be a proper subdomain. We call
$\Omega$ a {\it $\delta$-John domain}, $\delta>0$, if there is a
point $x_{0}\in\Omega$ which can be joined with any point
$x\in\Omega$ by a continuous curve $\gamma\subset\Omega$ which
satisfies $$\delta|\xi-x|\leq\,d_{\Omega}(\xi)$$
 for all $\xi\in\gamma$, where $d_{\Omega}(\xi)$ is the Euclidean
 distance between $\xi$ and the boundary of $\Omega$.
\end{Def}

Definition \ref{eq-de-1} is equivalent to the usual definition of a
John domain.  The $\delta$-John domains are wide class of domains
including quasidisks and bounded uniform domains (cf. \cite{N}).
Note that a $\delta$-John domain is bounded in $\mathbb{C}^n$ (see
\cite[p.273-p.274]{I-N}).

Iwaniec and Nolder \cite{I-N} established estimates similar to
Theorem B for the components of a quasiregular mapping in domains in
$\mathbb{R}^n$ which satisfy certain geometric conditions. On John
domains, Nolder \cite{N} obtained a generalization of Theorem B to
solutions of certain elliptic equations in divergence form, whose
gradients have comparable local $L^{\alpha}$ norms.

By analogy with Theorem B, we get the following result for
pluriharmonic functions on John domains.

\begin{Thm}\label{thm-2}
Let $f=h+\overline{g}$ be a pluriharmonic function in a  domain
$\Omega\subset\mathbb{C}^{n}$, where $h$ and $g$ are holomorphic in
$\Omega$ and $\Omega$ is a $\delta$-John domain. Then, for
$p\in(0,\infty)$, $q\in(1,\infty)$ and $M\in [1,\infty)$, there is a
positive constant $C$ depending only on $p$, $q$, $n$,   $\delta$
and $M$ such that
 $$\inf_{a\in \mathbb{C}}\|h-a\|_{p,\Omega,\omega}+\inf_{b\in \mathbb{C}}\|g-b\|_{p,\Omega,\omega}\leq C\|f\|_{p,\Omega,\omega},$$
 where $\omega\in A_{M}^{q}(\mathbb{C}^{n})$.
\end{Thm}

The following result  easily follows from Theorem \ref{thm-2}.
\begin{Cor}
Let $f=h+\overline{g}$ be a pluriharmonic function in a domain
$\Omega\subset\mathbb{C}^{n}$, where $h$ and $g$ are holomorphic in
$\Omega$ and $\Omega$ is a $\delta$-John domain. Then,for
$p\in(0,\infty)$, $q\in(1,\infty)$ and $M\in [1,\infty)$,
$\|f\|_{p,\Omega,\omega}<\infty$ if and only if
$\|h\|_{p,\Omega,\omega}+\|g\|_{p,\Omega,\omega}<\infty$, where
$\omega\in A_{M}^{q}(\mathbb{C}^{n})$.
\end{Cor}

By taking $g=-h$ in Theorem \ref{thm-2}, we obtain the following
generalization of Theorem B to higher dimensions.

\begin{Cor}\label{cor-2b}
Let $h=u+iv$ be a holomorphic function in a  domain
$\Omega\subset\mathbb{C}^{n}$, where $u$ and $v$ are real valued
pluriharmonic functions in $\Omega$ and $\Omega$ is a $\delta$-John
domain. Then, for $p\in(0,\infty)$, $q\in(1,\infty)$ and $M\in
[1,\infty)$, there is a positive constant $C$ depending only on $p$,
$q$, $n$,   $\delta$ and $M$ such that
 $$\inf_{a\in \mathbb{R}}\|u-a\|_{p,\Omega,\omega}\leq C\|v\|_{p,\Omega,\omega},$$
 where $\omega\in A_{M}^{q}(\mathbb{C}^{n})$.
\end{Cor}

In the case $p\geq 1$, we further obtain the following results.

\begin{Thm}\label{thm-2c}
Let $p\in [1,\infty)$, $q\in(1,\infty)$, $M\in [1,\infty)$, and
$\omega\in A_{M}^{q}(\mathbb{C}^{n})$. Let $f=h+\overline{g}$ be a
pluriharmonic function in a  domain $\Omega\subset\mathbb{C}^{n}$,
where $h, g\in L^{p}(\Omega,\omega)$ are holomorphic in $\Omega$ and
$\Omega$ is a $\delta$-John domain. Then, there is a positive
constant $C$ depending only on $p$, $q$, $n$,   $\delta$ and $M$
such that
 $$\|h-h_{\Omega, \omega}\|_{p,\Omega,\omega}+\|g-g_{\Omega, \omega}\|_{p,\Omega,\omega}\leq C\|f\|_{p,\Omega,\omega}.$$
\end{Thm}

By taking $g=-h$ in Theorem \ref{thm-2c}, we obtain the following
result.

\begin{Cor}\label{cor-2d}
Let $p\in [1,\infty)$, $q\in(1,\infty)$, $M\in [1,\infty)$, and
$\omega\in A_{M}^{q}(\mathbb{C}^{n})$. Let $h=u+iv\in
L^{p}(\Omega,\omega)$ be a holomorphic function in a  domain
$\Omega\subset\mathbb{C}^{n}$, where $u$ and $v$ are real valued
pluriharmonic functions in $\Omega$ and $\Omega$ is a $\delta$-John
domain. Then, there is a positive constant $C$ depending only on
$p$, $q$, $n$,   $\delta$ and $M$ such that
 $$\| u-u_{\Omega,\omega}\|_{p,\Omega, \omega}\leq C\|v\|_{p,\Omega,\omega}.$$
\end{Cor}

\begin{Rem}
Iwaniec and Nolder \cite[Definition 1]{I-N} introduced domains with
a chain condition $\Omega \in \mathscr{F}(\sigma, N)$ with
$\sigma>1$, $N\geq 1$. These classes contain many important types of
domains in the Euclidean space such as cubes, balls and John
domains. Every $\Omega \in \mathscr{F}(\sigma, N)$ is bounded.
Theorems \ref{thm-2}, \ref{thm-2c} and the above corollaries can be
generalized to domains $\Omega \in \mathscr{F}(\sigma, N)$ with
$\sigma>1$, $N\geq 1$.
\end{Rem}

\subsection{The Hardy-Littlewood type theorems and smooth moduli}
A continuous increasing function
$\psi:[0,\infty)\rightarrow[0,\infty)$ with $\psi(0)=0$ is called a
majorant if $\psi(t)/t$ is non-increasing for $t>0$ (see
\cite{Dy1,Dy2}). For $\delta_{0}>0$ and $0<\delta<\delta_{0}$, we
consider the following conditions on a majorant $\psi$:
\begin{equation}\label{eq2x} \int_{0}^{\delta}\frac{\psi(t)}{t}\,dt\leq C
\psi(\delta)
 \end{equation} and
 \begin{equation}\label{eq3x}
\delta\int_{\delta}^{\infty}\frac{\psi(t)}{t^{2}}\,dt\leq C
\psi(\delta), \end{equation}
 where $C$ denotes a positive
constant. A majorant $\psi$ is henceforth  called fast (resp. slow)
if condition (\ref{eq2x}) (resp. (\ref{eq3x}) ) is fulfilled. In
particular, a majorant $\psi$ is said to be  regular if it satisfies
the conditions (\ref{eq2x}) and (\ref{eq3x}) (see
\cite{Dy1,Dy2,Pa}).

Given a majorant $\psi$ and a proper subdomain  $\Omega^{\ast}$ of
$\mathbb{R}^m$ or $\mathbb{C}^{n}$, a function $f$ from
$\Omega^{\ast}$ into $\mathbb{C}$ or $\mathbb{R}:=(-\infty,\infty)$
is said to belong to the  Lipschitz  space
$\mathscr{L}_{\psi}(\Omega^{\ast})$
 if there is a positive constant $C:=C(f)$ such that

 \begin{equation}\label{rrt-1}|f(z)-f(w)| \leq\,C\psi\left(|z-w|\right), \quad z,w
\in \Omega^{\ast}.
\end{equation} Note that if $f\in
\mathscr{L}_{\psi}(\Omega^{\ast})$, then $f$ is continuous on
$\overline{\Omega^{\ast}}$ and (\ref{rrt-1}) holds for $z,w \in
\overline{\Omega^{\ast}}$ (see \cite{Dy2}).
 In particular, we say that
a function $f$ belongs to the  local Lipschitz space
$\mbox{loc}\mathscr{L}_{\psi}(\Omega^{\ast})$ if (\ref{rrt-1})
holds, with a fixed positive constant $C$, whenever $z\in
\Omega^{\ast}$ and $|z-w|<\frac{1}{2}d_{\Omega^{\ast}}(z)$ (cf.
\cite{Dy2,GM,L}). Moreover, $\Omega^{\ast}$ is called an
$\mathscr{L}_{\psi}$-extension domain if
$\mathscr{L}_{\psi}(\Omega^{\ast})=\mbox{loc}\mathscr{L}_{\psi}(\Omega^{\ast}).$
On the geometric characterization of $\mathscr{L}_{\psi}$-extension
domains, see
 \cite{GM}. In \cite{L}, Lappalainen
generalized the characterization of \cite{GM}, and proved that
$\Omega^{\ast}$ is a $\mathscr{L}_{\psi}$-extension domain if and
only if each pair of points $z,w\in \Omega^{\ast}$ can be joined by
a rectifiable curve $\gamma\subset \Omega^{\ast}$ satisfying

 \begin{equation}\label{eq1.0}
\int_{\gamma}\frac{\psi(d_{\Omega^{\ast}}(\zeta))}{d_{\Omega^{\ast}}(\zeta)}\,ds(\zeta)
\leq C\psi(|z-w|)
\end{equation} with some fixed positive constant
$C=C(\Omega^{\ast},\psi)$, where $ds$ stands for the arc length
measure on $\gamma$.  Furthermore, Lappalainen \cite[Theorem
4.12]{L} proved that $\mathscr{L}_{\psi}$-extension domains  exist
only for fast majorants $\psi$. In particular, $\mathbb{B}^{n}$ is
an $\mathscr{L}_{\psi}$-extension domain of $\mathbb{C}^n$ for fast
majorant $\psi$ (see \cite{Dy2}).

Given two sets $E_{1}$ and $E_{2}$ in $\mathbb{R}^m$ or
$\mathbb{C}^{n}$, we write $\mathscr{L}_{\psi}(E_{1},E_{2})$ for the
class of those continuous functions $f$ in $E_{1}\cup\,E_{2}$ which
satisfy (\ref{rrt-1}), with some positive constant $C=C(f)$,
whenever $z\in\,E_{1}$ and $w\in\,E_{2}$ (see \cite{Dy2}).

In \cite{Dy1}, Dyakonov gave some characterizations of the
holomorphic functions of the class $\mathscr{L}_{\psi}(\mathbb{D})$
in terms of their moduli. Let's recall one of the important findings
in \cite{Dy1} as follows (see also \cite[Theorem A]{Pa}).

\begin{ThmD}\label{Dya-1}
Let $\psi$ be a regular majorant and $f$ be a holomorphic function
in $\mathbb{D}$. Then $f\in\mathscr{L}_{\psi}(\mathbb{D})$ if and
only if $|f|\in\mathscr{L}_{\psi}(\mathbb{D})$.
\end{ThmD}

Later, Pavlovi\'c \cite{Pa} came up with a relatively simple proof
of Theorem D. By using Pavlovi\'c's method, in conjunction with
other techniques, Dyakonov \cite{Dy2}  generalized Theorem D into
the following form.

\begin{ThmE}{\rm (\cite[Theorem 1]{Dy2})}\label{Dya-2}
Let $\psi$ be a fast majorant, and let $\mathscr{E}$ be an
$\mathscr{L}_{\psi}$-extension domain of $\mathbb{C}^{n}$. Suppose
that $f$ is a holomorphic function  in $\mathscr{E}$. Then
$f\in\mathscr{L}_{\psi}(\mathscr{E})$$\Leftrightarrow$$|f|\in\mathscr{L}_{\psi}(\mathscr{E})$
$\Leftrightarrow$$|f|\in\mathscr{L}_{\psi}(\mathscr{E},\partial\mathscr{E})$.
\end{ThmE}

If $f=u+iv$ is a holomorphic function in $\mathbb{D}$, and if
$u={\rm Re}(f)\in\mathscr{L}_{\psi_{\alpha}}(\mathbb{D})$, then so
does $v$ (and hence $f$), with the same $\alpha\in  (0,1]$, where
$\psi_{\alpha}(t)=t^{\alpha}~(t\geq0)$. This fact, proved in
\cite{HL}, is known as the Hardy-Littlewood theorem on conjugate
functions. We remark, however, that in the case $\alpha\in(0,1)$ an
earlier-and essentially equivalent- result of Privalov \cite{P}
tells us that harmonic conjugation preserves the Lipschitz class
$\mathscr{L}_{\psi_{\alpha}}(\mathbb{T})$ of the circle
$\mathbb{T}:=\partial\mathbb{D}=\mathbb{S}^{0}$.
 In \cite{Dy2}, Dyakonov  extended
the Hardy-Littlewood theorem to the high-dimensional case as follows
(see \cite[Ineq. (2.3)]{Dy2}).

\begin{ThmF}\label{Dy-1997} Let $\psi$ be a fast majorant, and let $\mathscr{E}$ be an   $\mathscr{L}_{\psi}$-extension domain of $\mathbb{C}^{n}$.
If $u+iv$ is a bounded  holomorphic function in $\mathscr{E}$, then
$u\in\mathscr{L}_{\psi}(\mathscr{E})\Rightarrow\,v\in\mathscr{L}_{\psi}(\mathscr{E})$.
\end{ThmF}

Hardy and Littlewood  \cite{HL} proved the following theorem for
holomorphic functions with respect to the majorant
$\psi_{\alpha}(t)=t^{\alpha}~(t\geq0)$ as follows, where
$\alpha\in(0,1]$. For relevant research on harmonic functions, see
\cite{Kr,MM,Pa-2007}.

\begin{ThmG}{\rm (\cite[Theorem 40]{HL})}\label{Shiga}
Let $f$ be a holomorphic function in $\mathbb{D}$ and continuous up
to $\overline{\mathbb{D}}:=\mathbb{D}\cup\mathbb{T}$. Let
$\psi_{\alpha}(t)=t^{\alpha}~(t\geq0)$. Suppose that there exists
$\alpha\in(0,1]$ and a positive constant $C$ such that

$$|f(e^{i\theta_{1}})-f(e^{i\theta_{2}})|\leq\,C\psi_{\alpha}(|\theta_{1}-\theta_{2}|)~\mbox{for}~0\leq\theta_{1},~\theta_{2}<2\pi.$$
Then
$$|f'(z)|\leq\,C\frac{\psi_{\alpha}(d_{\mathbb{D}}(z))}{d_{\mathbb{D}}(z)}~\mbox{for}~z\in\mathbb{D}.$$
\end{ThmG}

In fact, there is an essential connection among Theorems E, F and G.
By using  new techniques, in conjunction with some methods of
Dyakonov and Pavlovi\'c, we remove ``the boundedness condition" in
Theorem F and establish extensions of Theorems E, F and G and also a
connection among Theorems E, F and G as follows.

\begin{Thm}\label{thm-3}
 Let $\psi$ be a fast majorant, and let $\mathscr{E}$ be an   $\mathscr{L}_{\psi}$-extension domain of $\mathbb{C}^{n}$.
If $f=u+iv$ is holomorphic in $\mathscr{E}$,
 then  the following statements are equivalent, where $u$ and $v$ are real-valued pluriharmonic functions in $\mathscr{E}$.
\begin{enumerate}
\item[{\rm ($\mathscr{C}_{1}$)}]  $u\in\mathscr{L}_{\psi}(\mathscr{E})$;
\item[{\rm ($\mathscr{C}_{2}$)}] There is a positive constant $C$ such that $$|\nabla\,u(z)|\leq\,C\frac{\psi(d_{\mathscr{E}}(z))}{d_{\mathscr{E}}(z)},
\quad z\in\mathscr{E},$$ where
$z=(z_{1},\ldots,z_{n})=(x_{1}+iy_{1},\ldots,x_{n}+iy_{n})$ and
$$\nabla\,u(z)=(u_{x_{1}},u_{y_{1}}\ldots,u_{x_{n}},u_{y_{n}});$$
\item[{\rm ($\mathscr{C}_{3}$)}]  $|u|\in\mathscr{L}_{\psi}(\mathscr{E})$;
\item[{\rm ($\mathscr{C}_{4}$)}] If $n\geq2$, then, for a given  $p\geq\frac{2n-2}{2n-1}$, there is a positive constant $C$ such that
\begin{equation}\label{eq-rfc-1}I_{u,p}(z)\leq\,C\frac{\psi(d_{\mathscr{E}}(z))}{d_{\mathscr{E}}(z)},~z\in\mathscr{E},
\end{equation} where
$I_{u,p}(z):=|\mathbb{B}^{n}(z,d_{\mathscr{E}}(z)/2)|^{-1/p}\|\nabla\,u\|_{p,\mathbb{B}^{n}(z,d_{\mathscr{E}}(z)/2)}$.
Moreover, if $n=1$, then, for a given $p>0$, there is a positive
constant $C$ such that {\rm(\ref{eq-rfc-1})} holds;

\item[{\rm ($\mathscr{C}_{5}$)}]  $v\in\mathscr{L}_{\psi}(\mathscr{E})$;
\item[{\rm ($\mathscr{C}_{6}$)}] There is a positive constant $C$ such that
 $$|\nabla\,v(z)|\leq\,C\frac{\psi(d_{\mathscr{E}}(z))}{d_{\mathscr{E}}(z)},~z\in\mathscr{E};$$
\item[{\rm ($\mathscr{C}_{7}$)}]  $|v|\in\mathscr{L}_{\psi}(\mathscr{E})$;
\item[{\rm ($\mathscr{C}_{8}$)}] If $n\geq2$, then, for a given  $p\geq\frac{2n-2}{2n-1}$, there is a positive constant $C$ such that
\begin{equation}\label{eq-rfc-2}I_{v,p}(z)\leq\,C\frac{\psi(d_{\mathscr{E}}(z))}{d_{\mathscr{E}}(z)},~z\in\mathscr{E}.
\end{equation} Moreover, if $n=1$, then, for a given  $p>0$, there is
a positive constant $C$ such that {\rm(\ref{eq-rfc-2})} holds;
\item[{\rm ($\mathscr{C}_{9}$)}] $f\in\mathscr{L}_{\psi}(\mathscr{E})$;
\item[{\rm ($\mathscr{C}_{10}$)}] $|f|\in\mathscr{L}_{\psi}(\mathscr{E})$;
\item[{\rm ($\mathscr{C}_{11}$)}] $|f|\in\mathscr{L}_{\psi}(\mathscr{E},\partial\mathscr{E})$.
\end{enumerate}
\end{Thm}

Let $f$ be a holomorphic function or a quasiconformal mapping. It
follows from \cite{Dy1,Dy2,Dy3} that the implication
$$|f|\in\mathscr{L}_{\psi}(\Omega_{1})\Rightarrow\,f\in\mathscr{L}_{\psi}(\Omega_{1})$$
depends heavily on the geometry of the proper subdomain $\Omega_{1}$
of $\mathbb{C}^{n}$ or $\mathbb{R}^m$.  One might still look for a
universal theorem, relating the properties of $|f|$ to those of $f$,
that should work for all domains $\Omega_{1}$ of $\mathbb{C}^n$
 and all holomorphic functions $f$ in $\Omega_{1}$. In order to produce such a statement,
 as in Dyakonov \cite{Dy2},
 we modify the Lipschitz condition
 by using a suitable ``internal distance" in $\Omega_{1}$,
 where $\Omega_1$ a proper subdomain of $\mathbb{C}^n$ or $\mathbb{R}^m$.
 For $z_{1},~z_{2}\in\Omega_{1}$, let
 $$d_{\psi,\Omega_{1}}(z_{1},~z_{2}):=\inf\int_{\gamma}\frac{\psi(d_{\Omega_{1}}(z))}{d_{\Omega_{1}}(z)}\,ds(z),$$
where $\psi$ is a majorant and the infimum is taken over all
rectifiable curves $\gamma\subset\Omega_{1}$ joining $z_{1}$ to
$z_{2}$. Further, we write
$F\in\mathscr{L}_{\psi,{\rm\,int}}(\Omega_{1})$ to mean that there
is a positive constant $C=C(F)$ such that

$$|F(z_{1})-F(z_{2})|\leq\,Cd_{\psi,\Omega_{1}}(z_{1},~z_{2})$$
whenever $z_{1},~z_{2}\in\Omega_{1}$, where  $F$ is a function of
$\Omega_{1}$ into $\mathbb{C}$ or $\mathbb{R}$.

In \cite[Theorem 3]{Dy2}, by replacing the Euclidean distance with
the internal distance, Dyakonov obtained the following result.

\begin{ThmH}{\rm (\cite[Theorem 3]{Dy2})}\label{Dya-3}
Let $\psi$ be a fast majorant, and let $f$ be a holomorphic function
in a  domain $\Omega_{1}\subset\mathbb{C}^{n}$.  Then
$f\in\mathscr{L}_{\psi,{\rm\,int}}(\Omega_{1})$$\Leftrightarrow$$|f|\in\mathscr{L}_{\psi,{\rm\,int}}(\Omega_{1})$.
\end{ThmH}

\begin{Rem}
In fact, in view of Proposition \ref{prop-1}, we should assume that
$\Omega_1$  is bounded in Theorem H.
\end{Rem}

\begin{Pro}\label{prop-1}
There exists a fast majorant $\psi_0$ such that there does not exist
a constant $C>0$ such that
\[
\int_{0}^{r}\frac{\psi_0(t)}{t}\,dt\leq C \psi_0(r), \quad \forall
r>0.
\]
\end{Pro}
\begin{proof} Let
\[
\psi_0(t)=\left\{
\begin{array}{ll}
\frac{1}{e}t, & 0\leq t\leq e, \\
\ln t, & t>e.
\end{array}
\right.
\]
Then $\psi_0$ is a majorant and
\[
\int_0^{\delta}\frac{\psi_0(t)}{t}\,dt=\frac{1}{e}\delta=\psi_0(\delta),
\quad 0<\delta<e.
\]
So, $\psi_0$ is a fast majorant. On the other hand, for $r>e$, we
have
\[
\int_{0}^{r}\frac{\psi_0(t)}{t}\,dt=1+\int_{e}^{r}\frac{\ln
t}{t}\,dt=\frac{1}{2}+\frac{1}{2}(\ln r)^2
>\frac{1}{2}(\ln r) \psi_0(r).
\]
The proof of this proposition is finished.
\end{proof}

By  replacing the Euclidean distance with the internal distance, we
establish extensions of Theorems F, G and H and also a connection
among Theorems F, G and H as follows.

\begin{Thm}\label{thm-4}
 Let $\psi$ be a fast majorant, and let $\Omega_{1}$ be a bounded domain of $\mathbb{C}^{n}$.
If $u$ and $v$ are real-valued pluriharmonic functions in
$\Omega_{1}$ with $f=u+iv$  holomorphic on $\Omega_1$, then  the
following statements are equivalent.
\begin{enumerate}
\item[{\rm ($\mathscr{D}_{1}$)}]  $u\in\mathscr{L}_{\psi,{\rm\,int}}(\Omega_{1})$;
\item[{\rm ($\mathscr{D}_{2}$)}]
There is a positive constant $C$ such that
$$|\nabla\,u(z)|\leq\,C\frac{\psi(d_{\Omega_{1}}(z))}{d_{\Omega_{1}}(z)},
\quad z\in\Omega_{1};$$
\item[{\rm ($\mathscr{D}_{3}$)}]  $|u|\in\mathscr{L}_{\psi,{\rm\,int}}(\Omega_{1})$;
\item[{\rm ($\mathscr{D}_{4}$)}] If $n\geq2$, then, for a given  $p\geq\frac{2n-2}{2n-1}$, there is a positive constant $C$ such that
\begin{equation}\label{eq-rfc-1-1}
I_{u,p}(z)\leq\,C\frac{\psi(d_{\Omega_{1}}(z))}{d_{\Omega_{1}}(z)},~z\in\Omega_{1}.
\end{equation}
Moreover, if $n=1$, then, for a given $p>0$, there is a positive
constant $C$ such that {\rm(\ref{eq-rfc-1-1})} holds;

\item[{\rm ($\mathscr{D}_{5}$)}]  $v\in\mathscr{L}_{\psi,{\rm\,int}}(\Omega_{1})$;
\item[{\rm ($\mathscr{D}_{6}$)}] There is a positive constant $C$ such that
 $$|\nabla\,v(z)|\leq\,C\frac{\psi(d_{\Omega_{1}}(z))}{d_{\Omega_{1}}(z)},~z\in\Omega_{1};$$
\item[{\rm ($\mathscr{D}_{7}$)}]  $|v|\in\mathscr{L}_{\psi,{\rm\,int}}(\Omega_{1})$;
\item[{\rm ($\mathscr{D}_{8}$)}] If $n\geq2$, then, for a given  $p\geq\frac{2n-2}{2n-1}$, there is a positive constant $C$ such that
\begin{equation}\label{eq-rfc-2-1}
I_{v,p}(z)\leq\,C\frac{\psi(d_{\Omega_{1}}(z))}{d_{\Omega_{1}}(z)},~z\in\Omega_{1}.
\end{equation}
Moreover, if $n=1$, then, for a given  $p>0$, there is a positive
constant $C$ such that {\rm(\ref{eq-rfc-2-1})} holds;
\item[{\rm ($\mathscr{D}_{9}$)}] $f\in\mathscr{L}_{\psi,{\rm\,int}}(\Omega_{1})$;
\item[{\rm ($\mathscr{D}_{10}$)}]
$|f|\in\mathscr{L}_{\psi,{\rm\,int}}(\Omega_{1})$.

\end{enumerate}
\end{Thm}

For real-valued harmonic functions in $\mathbf{B}^{n}$, where
$n\geq2$,  Pavlovi\'c proved the following result in \cite{Pa-2007}.

\begin{ThmI}{\rm (See the ``Abstract" of \cite{Pa-2007})}\label{P-07}
Let $u$ be a real-valued harmonic function  in $\mathbf{B}^{n}$ and
continuous on $\overline{\mathbf{B}^{n}}$, where $n\geq2$. Then,
 for $\alpha\in(0,1)$, the following conditions are equivalent:

\begin{enumerate}
\item[{\rm $\bullet$}] $|u(x)-u(y)|\leq\,C|x-y|^{\alpha}, ~~x,y\in\mathbf{B}^{n};$
\item[{\rm $\bullet$}] $\big||u(y)|-|u(\zeta)|\big|\leq\,C|y-\zeta|^{\alpha}, ~~y,\zeta\in\mathbf{S}^{n-1};$
\item[{\rm $\bullet$}] $\big||u(\zeta)|-|u(r\zeta)|\big|\leq\,C|1-r|^{\alpha}, ~~\zeta\in\mathbf{S}^{n-1},~r\in(0,1).$
\end{enumerate}
\end{ThmI}

By using different proof methods from that in  \cite{Pa-2007},
 we  extend Theorems D, E, G, H and I to real-valued harmonic functions on
 an $\mathscr{L}_{\psi}$-extension domain of $\mathbb{R}^{n}$ as follows, where $n\geq2$.

\begin{Thm}\label{Har-1}
Let $\psi$ be a fast majorant,  and let $\Omega_{2}$ be an
$\mathscr{L}_{\psi}$-extension domain of $\mathbb{R}^{n}$, where
$n\geq2$. If $u$ is a real-valued harmonic function in $\Omega_{2}$,
then  the following statements are equivalent.
\begin{enumerate}
\item[{\rm ($\mathscr{F}_{1}$)}] $u\in\mathscr{L}_{\psi}(\Omega_{2})$;
\item[{\rm ($\mathscr{F}_{2}$)}]
There is a positive constant $C$ such that
$$|\nabla\,u(x)|\leq\,C\frac{\psi(d_{\Omega_{2}}(x))}{d_{\Omega_{2}}(x)},
\quad x\in\Omega_{2};$$
\item[{\rm ($\mathscr{F}_{3}$)}] $|u|\in\mathscr{L}_{\psi}(\Omega_{2})$;
\item[{\rm ($\mathscr{F}_{4}$)}] $|u|\in\mathscr{L}_{\psi}(\Omega_{2},\partial\Omega_{2})$;
\item[{\rm ($\mathscr{F}_{5}$)}] If $n>2$, then, for a given  $p\geq\frac{n-2}{n-1}$, there is a positive constant $C$ such that
\begin{equation}\label{eq-rfj-1}\tilde{I}_{u,p}(x)\leq\,C\frac{\psi(d_{\Omega_{2}}(x))}{d_{\Omega_{2}}(x)},~x\in\Omega_{2},
\end{equation} where
$\tilde{I}_{u,p}(x):=|\mathbf{B}^{n}(x,d_{\Omega_2}(x)/2)|^{-1/p}\|\nabla\,u\|_{p,\mathbf{B}^{n}(x,d_{\Omega_2}(x)/2)}$.
Moreover, if $n=2$, then, for a given $p>0$, there is a positive
constant $C$ such that {\rm(\ref{eq-rfj-1})} holds.
\end{enumerate}
\end{Thm}

If  the Euclidean distance is replaced by
 the internal distance in Theorem \ref{Har-1}, then we obtain the following result.
 Here we omit the proof because it suffices to use (\ref{Ha-ch-1}), (\ref{Ha-ch-2}) and Theorems  N, O instead of (\ref{Poisson-1t}), (\ref{eq-4ct}) and Theorem M, respectively, and
 use arguments similar to those in the proof of Theorem \ref{thm-4}.

\begin{Thm}\label{Har-2}
Let $\psi$ be a fast majorant,  and let $\Omega_{3}$ be a   bounded
domain of $\mathbb{R}^{n}$, where $n\geq2$. If $u$ is a real-valued
harmonic function in $\Omega_{3}$, then  the following statements
are equivalent.
\begin{enumerate}
\item[{\rm ($\mathscr{G}_{1}$)}] $u\in\mathscr{L}_{\psi,{\rm\,int}}(\Omega_{3})$;
\item[{\rm ($\mathscr{G}_{2}$)}]
There is a positive constant $C$ such that
$$|\nabla\,u(x)|\leq\,C\frac{\psi(d_{\Omega_{3}}(x))}{d_{\Omega_{3}}(x)},
\quad x\in\Omega_{3};$$
\item[{\rm ($\mathscr{G}_{3}$)}]
$|u|\in\mathscr{L}_{\psi,{\rm\,int}}(\Omega_{3})$;
\item[{\rm ($\mathscr{G}_{4}$)}] If $n>2$, then, for a given  $p\geq\frac{n-2}{n-1}$, there is a positive constant $C$ such that
\begin{equation}\label{eq-rfj-2}\tilde{I}_{u,p}(x)\leq\,C\frac{\psi(d_{\Omega_{3}}(x))}{d_{\Omega_{3}}(x)},~x\in\Omega_{3}.
\end{equation}
Moreover, if $n=2$, then, for a given $p>0$, there is a positive
constant $C$ such that {\rm(\ref{eq-rfj-2})} holds.
\end{enumerate}
\end{Thm}

The proofs of Theorems \ref{thm-1}, \ref{thm-2} and \ref{thm-2c}
will be presented in section \ref{sec3}, and the proof of Theorems
\ref{thm-3}, \ref{thm-4} and \ref{Har-1} will be given in section
\ref{sec4}.

\section{The Riesz type inequalities and the Hardy-Littlewood type theorems}\label{sec3}

The following results will play an important role in the proof of
Theorem \ref{thm-1}.

\begin{LemJ}{\rm (cf. \cite[Lemma 5]{CPR})}\label{Lemx}
Suppose that $x,~y\in[0,\infty)$ and $\tau\in(0,\infty)$. Then
$$(x+y)^{\tau}\leq2^{\max\{\tau-1,0\}}(x^{\tau}+y^{\tau}).$$
\end{LemJ}



\begin{ThmK}{\rm (\cite[Theorem 7]{H-L})}\label{HL}
Suppose that $f=u+iv$ is a holomorphic function in $\mathbb{D}$ with $v(0)=0$. 
If  $u\in\mathscr{PH}^{p}(\mathbb{D})$ for some $0<p\leq1$, then
$v$ satisfies

\be\label{eq-Har-Lit}M_{p}(r,v)\leq\,C\|u\|_{p}+C\|u\|_{p}\left(\log\frac{1}{1-r}\right)^{\frac{1}{p}},\ee
where $C$ is a positive constant depending only on $p$.
\end{ThmK}

\begin{ThmL}{\rm (\cite[Theorem 5]{M})}\label{MK}
Suppose that $\Omega\subset\mathbb{C}^{n}$ is a bounded symmetric
domain with origin and Bergman-Silov boundary $b$. Let $f=u+iv$ be a
holomorphic function in $\Omega$. If $u\in\mathscr{PH}^{1}(\Omega)$,
then $v\in\mathscr{PH}^{q}(\Omega)$ for all $q\in(0,1)$.
\end{ThmL}

\subsection{The proof of Theorem \ref{thm-1}} We first prove $(\mathscr{B}_{1})$. Let
$f=h+\overline{g}=u+iv\in\mathscr{PH}^{p}(\Omega)$, where
$h=u_{1}+iv_{1}$ and $g=u_{2}+iv_{2}$. Then
$u,~v\in\mathscr{PH}^{p}(\Omega)$.  Set $F=h+g$ and $v^{\ast}={\rm
Im}(F)$, where $``{\rm Im}"$ denotes the imaginary part of a complex
number. It is not difficult to know that ${\rm Re}(F)={\rm
Re}(f)=u\in\mathscr{PH}^{p}(\Omega)$ and by using Fubini's theorem,
we obtain
\[
M_{p}^{p}(r,u)=\int_{b}|u(r\zeta)|^{p}d\sigma(\zeta)=\int_{b}M_{p,\zeta}^{p}(r,u)d\sigma(\zeta),
\]
where \beq\label{eq-chh-1b-new}
M_{p,\zeta}^{p}(r,u):=\frac{1}{2\pi}\int_{0}^{2\pi}|u(re^{i\theta}\zeta)|^{p}d\theta.\eeq
Since $|u|^{p}$ is plurisubharmonic on $\Omega$ (see \cite{Vl}),
$|U(z)|^{p}:=|u(z\zeta)|^{p}$ is subharmonic with respect to
$z\in\mathbb{D}$. Thus, by (\ref{eq-chh-1b-new}),  we see that
$M_{p,\zeta}^{p}(r,u)$ is increasing on $r\in[0,1)$ for  all
$\zeta\in\, b$, which, together with Lebesgue's monotone convergence
Theorem, yields that
\beq\label{eq-chh-2}\|u\|_{p}=\left(\int_{b}\lim_{r\to 1^{-}}
M_{p,\zeta}^{p}(r,u)d\sigma(\zeta)\right)^{\frac{1}{p}}<\infty.\eeq
Without loss of generality, we assume that $v_{1}(0)=v_{2}(0)=0$.
Since $V(z):=v^{\ast}(z\zeta)$ is the harmonic conjugate of $U(z)$
on $z\in\mathbb{D}$, by (\ref{Riez-2}), we see that

\be\label{eq-chh-3}
M_{p,\zeta}(r,v^{\ast})\leq\cot\frac{\pi}{2p^{\ast}}M_{p,\zeta}(r,u)
\ee for all $r\in (0,1)$ and $\zeta\in\,b$. By raising both sides of
(\ref{eq-chh-3}) to the $p$-th power and integrating over $b$, we
have

\beq\label{eq-chh-4}
\|v^{\ast}\|_{p}&=&\sup_{r\in[0,1)}M_{p}(r,v^{\ast}) \leq
\cot\frac{\pi}{2p^{\ast}}\sup_{r\in[0,1)}M_{p}(r,u)\\ \nonumber
&=&\cot\frac{\pi}{2p^{\ast}}\|u\|_{p}<\infty. \eeq

It follows from   Lemma J that

$$|v_{2}|^{p}=|{\rm Im}(g)|^{p}=\frac{|v^{\ast}-v|^{p}}{2^{p}}\leq\frac{|v^{\ast}|^{p}+|v|^{p}}{2},$$
which, together with (\ref{eq-chh-4}), yields that \beqq
\|v_{2}\|_{p}=\sup_{r\in[0,1)}M_{p}(r,v_{2})\leq2^{-\frac{1}{p}}\left(\|v^{\ast}\|_{p}+\|v\|_{p}\right)<\infty.
\eeqq

For $\xi\in\Omega$, let
$$G(\xi)=-ig(\xi)+iu_{2}(0)=v_{2}(\xi)+i(u_{2}(0)-u_{2}(\xi)).$$ By
using similar reasoning as in the proof of (\ref{eq-chh-4}), we have
$$\|{\rm Im}(G)\|_{p}\leq\cot\frac{\pi}{2p^{\ast}}\|v_{2}\|_{p}<\infty,$$
which implies that $g\in\mathscr{PH}^{p}(\Omega)$. Another desired
conclusion $h\in\mathscr{PH}^{p}(\Omega)$ follows from the following
inequality

$$|h|^{p}=|f-\overline{g}|^{p}\leq2^{p-1}(|f|^{p}+|g|^{p}).$$

Now we prove ($\mathscr{B}_{1}{\rm(I)}$). Since
$$|\mathcal{Q}(z)|^{p}:=\left(|h(z\zeta)|^{2}+|g(z\zeta)|^{2}\right)^{\frac{p}{2}}$$ is subharmonic with respect to
$z\in\mathbb{D}$ (see e.g. \cite[Lemma 2.13]{K-2019}),  we see that
$M_{p,\zeta}^{p}\left(r,\left(|h|^{2}+|g|^{2}\right)^{\frac{1}{2}}\right)$
is increasing on $r\in[0,1)$ for   all $\zeta\in\, b$, which,
together with $(\mathscr{B}_{1})$, Lemma J  and Lebesgue's monotone
convergence Theorem, yields that
\be\label{eq-chh-2-1}\sup_{r\in[0,1)}\left(\int_{b}M_{p,\zeta}^{p}(r,\mathcal{Q})d\sigma(\zeta)\right)^{\frac{1}{p}}=
\left(\int_{b}\lim_{r\to
1^{-}}M_{p,\zeta}^{p}(r,\mathcal{Q})d\sigma(\zeta)\right)^{\frac{1}{p}}<\infty.\ee
By Theorem C ($\mathscr{A}_{1}$), we have

\be\label{eq-chh-3t}
M_{p,\zeta}(r,\mathcal{Q})\leq\frac{1}{C_{1}(p)}M_{p,\zeta}(r,f) \ee
for  all $r\in(0,1)$ and $\zeta\in\, b$. Raise both sides of
(\ref{eq-chh-3t}) to the $p$-th power and integrate over $b$, which,
together with
 (\ref{eq-chh-2-1}), gives that

$$\lim_{r\to 1^{-}}\left(\int_{b}\big(|h(r\zeta)|^{2}+
|g(r\zeta)|^{2}\big)^{\frac{p}{2}}d\sigma(\zeta)\right)^{\frac{1}{p}}\leq
\frac{1}{C_{1}(p)} \|f\|_{p}.$$ The conclusion of
($\mathscr{B}_{1}{\rm(II)}$) follows from the similar reasoning as
in the proof of ($\mathscr{B}_{1}{\rm(I)}$).

Next, we prove ($\mathscr{B}_{2}$).  Since
$f=h+\overline{g}=u+iv\in\mathscr{PH}^{1}(\Omega)$, we see that
$u,~v\in\mathscr{PH}^{1}(\Omega)$. By Theorem L, we know that
$v^{\ast}={\rm Im}(F)={\rm Im}(h+g)\in\mathscr{PH}^{q}(\Omega)$ for
all $q\in(0,1)$. Consequently,
\[
|v_{2}|^{q}=|{\rm
Im}(g)|^{q}=\frac{|v^{\ast}-v|^{q}}{2^{q}}\leq\frac{|v^{\ast}|^{q}+|v|^{q}}{2^{q}},
\]
which, together with
$$\sup_{r\in[0,1)}\int_{b}|v(r\zeta)|^{q}d\sigma(\zeta)\leq
\sup_{r\in[0,1)}\left(\left(\int_{b}|v(r\zeta)|d\sigma(\zeta)\right)^{q}\left(\int_{b}d\sigma(\zeta)\right)^{1-q}\right)<\infty,$$
 gives that $v_{2}\in\mathscr{PH}^{q}(\Omega)$ for all $q\in(0,1)$. Let $F_{1}=h-g$ and $u^{\ast}={\rm Re}(F_{1}).$
Then
$${\rm Re}(-iF_{1})={\rm Im}(F_{1})={\rm Im}(f)\in\mathscr{PH}^{1}(\Omega),$$
which, together with Theorem L, implies that

$$-u^{\ast}={\rm Im}(-iF_{1})\in\mathscr{PH}^{q}(\Omega)$$ for all $q\in(0,1)$.
Hence

\[
|u_{2}|^{q}=|{\rm
Re}(g)|^{q}=\frac{|u-u^{\ast}|^{q}}{2^{q}}\leq\frac{|u|^{q}+|u^{\ast}|^{q}}{2^{q}},
\]
which, together with
$$\sup_{r\in[0,1)}\int_{b}|u(r\zeta)|^{q}d\sigma(\zeta)\leq
\sup_{r\in[0,1)}\left(\left(\int_{b}|u(r\zeta)|d\sigma(\zeta)\right)^{q}\left(\int_{b}d\sigma(\zeta)\right)^{1-q}\right)<\infty,$$
yields that $u_{2}\in\mathscr{PH}^{q}(\Omega)$ for all $q\in(0,1)$.
Combining the results $u_{2}, v_{2} \in\mathscr{PH}^{q}(\Omega)$ for
all $q\in(0,1)$  and  Lemma J gives that
$g=u_{2}+iv_{2}\in\mathscr{PH}^{q}(\Omega)$ for all $q\in(0,1)$.
 Thus, by Lemma J,  we see that $h=f-\overline{g}$ also belongs to $\mathscr{PH}^{q}(\Omega)$ for all $q\in(0,1)$.

Now we prove the second part of ($\mathscr{B}_{2}$). Using the above
notations, we may assume that $u_1(0)=u_2(0)=v_1(0)=v_2(0)=0$. Since
$M_{1,\zeta}(r,u)$ is increasing on $r\in[0,1)$ for  all $\zeta\in\,
b$, by (\ref{eq-chh-2}), we see that $M_{1,\zeta}(1,u)<\infty$ for
almost all $\zeta\in\, b$. Hence it follows from  (\ref{eq-Har-Lit})
that there is an absolute constant $C>0$ such that, for all
$r\in[0,1)$ and almost all $\zeta\in\, b$,

\be\label{eq-chh-8} M_{1,\zeta}(r,v^{\ast})\leq\,
CM_{1,\zeta}(1,u)\left(1+\log\frac{1}{1-r}\right). \ee By
integrating both sides of (\ref{eq-chh-8}) over $b$, we obtain

\be\label{eq-chh-9} M_{1}(r,v^{\ast})\leq\,
C\|u\|_{1}\left(1+\log\frac{1}{1-r}\right)\leq\,
C\|f\|_{1}\left(1+\log\frac{1}{1-r}\right). \ee

Since $$|v_{2}|=|{\rm
Im}(g)|=\frac{|v^{\ast}-v|}{2}\leq\frac{|v^{\ast}|+|v|}{2},$$ by
(\ref{eq-chh-9}), we see that

\beq\label{eqv-0}
M_{1}(r,v_{2})&\leq&\frac{1}{2}\left(\int_{b}|v^{\ast}(r\zeta)|d\sigma(\zeta)+\int_{b}|v(r\zeta)|d\sigma(\zeta)\right)\\
\nonumber
&\leq&\frac{1}{2}\left(M_{1}(r,v^{\ast})+\|f\|_{1}\right)\\
\nonumber
&\leq&\frac{\|f\|_{1}}{2}\left(1+C+C\log\frac{1}{1-r}\right). \eeq
Let $F_{1}^{\ast}=-i(h-g)$ and $\widetilde{u}={\rm
Im}(F_{1}^{\ast})$. Then
\[
{\rm Re}(F_{1}^{\ast})={\rm Im}(f)=v\in\mathscr{PH}^{1}(\Omega),
\]
which, together with   the similar reasoning as in the proof of
$M_{1}(r,v^{\ast})$, implies that there is an absolute constant
$C>0$ such that

\be\label{eq-4v}
M_{1}(r,\widetilde{u})\leq\,C\|f\|_{1}\left(1+\log\frac{1}{1-r}\right).
\ee Elementary calculations lead to

$$|u_{2}|=|{\rm Re}(g)|=\frac{|u+\widetilde{u}|}{2}\leq\frac{|u|+|\widetilde{u}|}{2},$$
which, together with (\ref{eq-4v}), gives that \beq\label{eq-5v}
M_{1}(r,u_{2})&\leq&\frac{1}{2}\left(\int_{b}|\widetilde{u}(r\zeta)|d\sigma(\zeta)+\int_{b}|u(r\zeta)|d\sigma(\zeta)\right)\\
\nonumber &\leq&2^{-1}M_{1}(r,\widetilde{u})+2^{-1}\|f\|_{1}\\
\nonumber &\leq&2^{-1}C\|f\|_{1}\left(1+
\log\frac{1}{1-r}\right)+2^{-1}\|f\|_{1}. \eeq From (\ref{eqv-0})
and (\ref{eq-5v}), we conclude that there is an absolute constant
$C>0$ such that

 \beq\label{eq-6v}
M_{1}(r,g) &\leq&\,M_{1}(r,u_{2})+M_{1}(r,v_{2})\\ \nonumber
&\leq&\,C\|f\|_{1}\left(1+\log\frac{1}{1-r}\right)+\|f\|_{1}. \eeq
On the other hand, by $h=f-\overline{g}$ and (\ref{eq-6v}), we see
that there is an absolute constant $C>0$ such that \beqq
M_{1}(r,h)&\leq&M_{1}(r,f)+M_{1}(r,g)\\ \nonumber
&\leq&2\|f\|_{1}+C\|f\|_{1}\left(1+\log\frac{1}{1-r}\right). \eeqq

Next, we prove the sharpness part of (\ref{Ch-Ha-1}). Let $\Omega$
be the unit polydisk. Consider the functions
$$h(z)=g(z)=\frac{1}{1-z_{j_{0}}},~z\in\Omega,$$ where $j_{0}\in\{1,\ldots,n\}$.
It is not difficult to know that $f=h+\overline{g}=2{\rm
Re}(h)\in\mathscr{PH}^{1}(\Omega)$ (see also  \cite[Chapter 3]{Du1}
or \cite{H-L}). However,    by \cite[Theorem E]{CLW-2019}, we have
\beqq
M_{1}(r,h)&=&\left(\frac{1}{2\pi}\right)^{n}\int_{0}^{2\pi}\cdots\int_{0}^{2\pi}
|h(re^{i\theta_{1}},\ldots,re^{i\theta_{n}})|d\theta_{1}\cdots\,d\theta_{n}\\
&=&\frac{1}{2\pi}\int_{0}^{2\pi}\frac{d\theta_{j_{0}}}{|1-re^{i\theta_{j_{0}}}|}
=\sum_{k=1}^{\infty}\left(\frac{\Gamma\left(n+\frac{1}{2}\right)}{n!\Gamma\left(\frac{1}{2}\right)}\right)^{2}r^{2n}\\
&\sim&\log\frac{1}{1-r}. \eeqq

At last,  we prove ($\mathscr{B}_{3}$). Suppose that $z^0\in
\partial \Omega$ be fixed such that
\[
h_0(z^0,z^0)=\max_{z\in \overline{\Omega}}h_0(z,z),
\]
where $h_0$ denotes the Bergman metric of $\Omega$ at $0$. Then,
for $z\in\Omega$, let $$f(z)=h(z)+\overline{g(z)},$$ where
\[
h(z)=g(z)=i\log\frac{h_0(z^0,z^0)+h_0(z,z^0)}{h_0(z^0,z^0)-h_0(z,z^0)}.
\]
By \cite[Theorem 6.5]{Lo}, $z^0\in b$. Also, it is not difficult to
know that $h,~g\notin\mathscr{PH}^{\infty}(\Omega)$ and
$f\in\mathscr{PH}^{\infty}(\Omega).$ The proof of this theorem is
finished. \qed

\subsection{The proof of Theorem \ref{thm-2}}
If $\|f\|_{p,\Omega,\omega}=\infty$, then it is obvious. Without
loss of generality, we assume that $\|f\|_{p,\Omega,\omega}<\infty$.
Let $f=h+\overline{g}=u+iv$, where $h=u_{1}+iv_{1}$ and
$g=u_{2}+iv_{2}$. Set $F=h+g$. Then ${\rm Re}(F)={\rm Re}(f)$. Since
$\|f\|_{p,\Omega,\omega}<\infty,$ we see that both
$\|u\|_{p,\Omega,\omega}$ and $\|v\|_{p,\Omega,\omega}$ are finite.
This gives $\|{\rm
Re}(F)\|_{p,\Omega,\omega}=\|u\|_{p,\Omega,\omega}\leq
\|f\|_{p,\Omega,\omega}<\infty.$
We have

$$ \inf_{a\in\mathbb{R}}\|{\rm
Re}(F)-a\|_{p,\Omega,\omega}\leq \|{\rm
Re}(F)\|_{p,\Omega,\omega}\leq \|f\|_{p,\Omega,\omega}<\infty.
$$
It follows from the Cauchy Riemann equations and \cite[Theorem
3.1]{N} that there is a positive constant $C$ depending only on $p$,
$q$, $n$,   $\delta$ and $M$ such that
$$\inf_{a\in\mathbb{R}}\|{\rm Im}(F)-a\|_{p,\Omega,\omega}\leq\, C\inf_{a\in\mathbb{R}}\|{\rm Re}(F)-a\|_{p,\Omega,\omega}
\leq C\|f\|_{p,\Omega,\omega}<\infty.$$
 Without loss of generality, let $$\|{\rm Im}(F)-a_{0}\|_{p,\Omega,\omega}
 =\inf_{a\in\mathbb{R}}\|{\rm Im}(F)-a\|_{p,\Omega,\omega}.$$
Since $${\rm Im}(g)=v_{2}=\frac{\mbox{\rm Im}(F)-v}{2},$$ we see
that

\begin{eqnarray*}
 \left\|{\rm
 Im}(g)-\frac{a_0}{2}\right\|_{p,\Omega,\omega}&=&\frac{1}{2}\left(\int_{\Omega}|{\rm Im}(F(z))-a_0-v(z)|^{p}\omega(z)d\mu(z)\right)^{\frac{1}{p}}\\
 &\leq&\frac{1}{2}C_0\left(\|{\rm Im}(F)-a_0\|_{p,\Omega,\omega}+\|v\|_{p,\Omega,\omega}\right)\\
 &\leq & \frac{1}{2}C_0(C+1)\|f\|_{p,\Omega,\omega},
 \end{eqnarray*}
 where
$C_{0}=2^{\max\left\{1-\frac{1}{p},\frac{1}{p}-1\right\}}.$ Using
the similar method, we obtain that there exists a constant $a_1\in
\mathbb{R}$ such that
$$\|{\rm Re}(g)-a_1\|_{p,\Omega,\omega}
\leq C\left\|{\rm
 Im}(g)-\frac{a_0}{2}\right\|_{p,\Omega,\omega}
 \leq
 \frac{1}{2}C_0C(C+1)\|f\|_{p,\Omega,\omega}.$$
Consequently, by Lemma J, we have

\beq \label{eq-ch-1} \left\|g-a_1-\frac{a_0}{2}i
\right\|_{p,\Omega,\omega} &\leq& C_{0} \bigg(\|{\rm
Re}(g)-a_1\|_{p,\Omega,\omega}\\ \nonumber
&&+\Big\|{\rm
Im}(g)-\frac{a_0}{2}\Big\|_{p,\Omega,\omega}\bigg)
\\
\nonumber &\leq & \frac{1}{2}C_0^2(C+1)^2\|f\|_{p,\Omega,\omega}.
\eeq

Since $h=f-\overline{g}$, by (\ref{eq-ch-1}) and Lemma J, we see
that

\beq\label{eq-ch-2} \left\|h+a_1-\frac{a_0}{2}i
\right\|_{p,\Omega,\omega} &\leq&
C_0\left(\|f\|_{p,\Omega,\omega}+\left\|g-a_1-\frac{a_0}{2}i
\right\|_{p,\Omega,\omega}\right)\\ \nonumber
&\leq&
C\|f\|_{p,\Omega,\omega}. \eeq Combining  (\ref{eq-ch-1}) and
(\ref{eq-ch-2}) gives the desired result. \qed

\subsection{The proof of Theorem \ref{thm-2c}}

We use the following lemma to prove Theorem \ref{thm-2c}.

\begin{Lem}\label{Lemx-3}
Let $\Omega$ be  a bounded domain in $\mathbb{C}^n$ and let
$\omega\in L_{loc}^{1}(\mathbb{C}^{n})$ with $\omega>0$ a. e..
Assume that $p\in[1,\infty)$. Then for complex valued functions
$f\in L^p(\Omega, \omega)$, we have $|f_{\Omega,\omega}|<\infty$ and
 \beq\label{eq-3-3-2}
\| f-f_{\Omega,\omega}\|_{p,\Omega, \omega}\leq
2^{2-\frac{1}{p}}\inf_{a\in \mathbb{C}}\|f-a\|_{p,\Omega,\omega}.
\eeq Also, for real valued functions $u\in L^p(\Omega, \omega)$, we
have $|u_{\Omega,\omega}|<\infty$ and
\[
\| u-u_{\Omega,\omega}\|_{p,\Omega, \omega}\leq
2^{2-\frac{1}{p}}\inf_{a\in \mathbb{R}}\|u-a\|_{p,\Omega,\omega}.
\]
\end{Lem}

\begin{proof}
We give a proof for $|f_{\Omega,\omega}|<\infty$ and that for
(\ref{eq-3-3-2}). The proof for the second part is similar. Since
$f\in L^p(\Omega, \omega)$, $p\in[1,\infty)$,  $\Omega$ is bounded
and $\omega\in L_{loc}^{1}(\mathbb{C}^{n})$ with $\omega>0$ a.e., we
obtain $|f_{\Omega,\omega}|<\infty$.

Without loss of generality, let
$$\|f-a_{0}\|_{p,\Omega,\omega}
 =\inf_{a\in\mathbb{C}}\|f-a\|_{p,\Omega,\omega}.$$
 Then, by Lemma J, we have
\begin{eqnarray*}
\| f-f_{\Omega, \omega}\|_{p,\Omega, \omega} &=& \|
(f-a_0)+(a_0-f_{\Omega, \omega})\|_{p,\Omega, \omega}
\\
&\leq& 2^{1-\frac{1}{p}}\left(\| f-a_0\|_{p,\Omega, \omega}
+\|a_0-f_{\Omega, \omega}\|_{p,\Omega, \omega}\right).
\end{eqnarray*}
Since
\begin{eqnarray*}
\|a_0-f_{\Omega, \omega}\|_{p,\Omega, \omega} &=& |a_0-f_{\Omega,
\omega}| \left(\int_{\Omega}\omega(z)d\mu(z)\right)^{\frac{1}{p}}
\\
&\leq&
\int_{\Omega}|a_0-f(z)|\omega(z)d\mu(z)\left(\int_{\Omega}\omega(z)d\mu(z)\right)^{\frac{1}{p}-1}
\\
&\leq & \| a_0-f\|_{p,\Omega,\omega},
\end{eqnarray*}
where the last inequality in the case $p>1$ follows from the
H\"{o}lder inequality. Thus, we have
\[
\| f-f_{\Omega, \omega}\|_{p,\Omega, \omega} \leq
2^{2-\frac{1}{p}}\| f-a_0\|_{p,\Omega, \omega}.
\]
This completes the proof.
\end{proof}

Theorem \ref{thm-2c} follows from Theorem \ref{thm-2} and Lemma
\ref{Lemx-3}. \qed

\section{The  Hardy-Littlewood type theorems and smooth moduli}\label{sec4}

Let's  begin this section with the following result.

\begin{ThmM}{\rm (\cite[Theorem  2.3]{K-2017})}\label{ThmF}
Let $u$ be a pluriharmonic function of $\mathbb{B}^{n}$ into
$(-1,1)$. Then, for $z\in\mathbb{B}^{n}$, the following sharp
inequality holds:
$$|\nabla\,u(z)|\leq\frac{4}{\pi}\frac{1-u^{2}(z)}{1-|z|^{2}}.$$
\end{ThmM}

\subsection{The proof of Theorem \ref{thm-3}}
We first prove
($\mathscr{C}_{1}$)$\Leftrightarrow$($\mathscr{C}_{2}$). Let's begin
to prove ($\mathscr{C}_{1}$)$\Rightarrow$($\mathscr{C}_{2}$).
 For any fixed $z=(z_{1},\ldots, z_{n}) \in\mathscr{E}$, let
$r=2d_{\mathscr{E}}(z)/3$. By the assumption, for
$w\in\overline{\mathbb{B}^{n}(z,r)},$ we see that there is a
positive constant $C$ such that \be\label{eq-1t}|u(z)-u(w)|\leq\,
C\psi(|z-w|)\leq\,C\psi(r).\ee For
$w=(w_{1},\ldots,w_{n})\in\mathbb{B}^{n}(z,r),$ it follows from the
invariant Poisson integral formula for real valued pluriharmonic
functions that \be\label{Poisson-1t} u(w)=\int_{\partial
\mathbb{B}^n}P_{r}(w-z,\xi)u(z+r\xi)d\sigma(\xi),\ee where
$P_{r}(w-z,\xi)=\frac{\left(r^{2}-|w-z|^{2}\right)^{n}}{\left|r-\langle
w-z,\xi\rangle\right|^{2n}}$.
 For $j\in\{1,\ldots,n\},$ elementary calculations lead to

\beq\label{eq-3ct} \left|\frac{\partial}{\partial
w_{j}}P_{r}(w-z,\xi)\right|&=&\bigg|
\frac{-n(\overline{w}_{j}-\overline{z}_{j})\left(r^{2}-|w-z|^{2}\right)^{n-1}}{\left|r-\langle
w-z,\xi\rangle\right|^{2n}}\\ \nonumber &&+
\frac{n\overline{\xi}_{j}\left(r^{2}-|w-z|^{2}\right)^{n}(r-\overline{\langle
w-z,\xi\rangle})}{\left|r-\langle
w-z,\xi\rangle\right|^{2n+2}}\bigg|\\ \nonumber
&\leq&\frac{n|w_{j}-z_{j}|\left(r^{2}-|w-z|^{2}\right)^{n-1}}{\left|r-\langle
w-z,\xi\rangle\right|^{2n}}\\ \nonumber
&&+\frac{n\left(r^{2}-|w-z|^{2}\right)^{n}|r-\overline{\langle
w-z,\xi\rangle}|}{\left|r-\langle w-z,\xi\rangle\right|^{2n+2}}.
\eeq It follows from (\ref{eq-3ct}) that, for
$w\in\mathbb{B}^{n}(z,3r/4)$ and $j\in\{1,\ldots,n\}$,

\beq\label{eq-4ct} \left|\frac{\partial}{\partial
w_{j}}P_{r}(w-z,\xi)\right| &\leq&\,n
\bigg(\frac{\frac{3}{4}r^{2n-1}}{\left(r-|w-z|\right)^{2n}}+\frac{\frac{7}{4}r^{2n+1}}{\left(r-|w-z|\right)^{2n+2}}\bigg)\\
\nonumber &\leq&\,\frac{115\cdot4^{2n-1}n}{r}. \eeq


Combining  (\ref{eq-1t}) and (\ref{eq-4ct}) gives
\beqq\label{eq-fff-1} \nonumber | u_{w}(w)|
 \nonumber&\leq&\left(\sum_{j=1}^{n}\bigg(\int_{\partial \mathbb{B}^n}\bigg|\frac{\partial}{\partial w_{j}}P_{r}(w-z,\xi)\bigg||u(z+r\xi)-u(z)|d\sigma(\xi)\bigg)^{2}\right)^{\frac{1}{2}}\\  \nonumber
 &\leq&\frac{115\cdot4^{2n-1}n\sqrt{n}}{r}\int_{\partial \mathbb{B}^n}|(u(z+r\xi)-u(z))|d\sigma(\xi)\\   \nonumber
 &\leq&115\cdot4^{2n-1}n\sqrt{n}C\frac{ \psi(r)}{r}
 \eeqq
and \be\label{jj-1}| u_{\overline{w}}(w)|=|
u_{w}(w)|\leq115\cdot4^{2n-1}n\sqrt{n}C\frac{ \psi(r)}{r},\ee where
$w\in\mathbb{B}^{n}(z,3r/4)$,  $u_{w}=(u_{w_{1}},\ldots,u_{w_{n}})$
and
$u_{\overline{w}}=(u_{\overline{w}_{1}},\ldots,u_{\overline{w}_{n}})$.
Taking $w=z$ in (\ref{jj-1}), we obtain \beq\label{eq-2t}
|\nabla\,u(z)|&=&| u_{w}(z)|+| u_{\overline{w}}(z)|\leq\,C_{1}\frac{
\psi(r)}{r}=\frac{3C_{1}}{2}\frac{
\psi\left(\frac{2d_{\mathscr{E}}(z)}{3}\right)}{d_{\mathscr{E}}(z)}\\\nonumber
&\leq&\,\frac{3C_{1}}{2}\frac{
\psi(d_{\mathscr{E}}(z))}{d_{\mathscr{E}}(z)}, \eeq where
$C_{1}=230\cdot4^{2n-1}n\sqrt{n}C$.


Now, we prove ($\mathscr{C}_{2}$)$\Rightarrow$($\mathscr{C}_{1}$).
Since $\mathscr{E}$ is an $\mathscr{L}_{\psi}$-extension domain, we
see that for all $z_{1},~z_{2}\in\Omega$, by using (\ref{eq1.0}),
there is a positive constant $C$ which does not depend on $z_1$,
$z_2$ and a rectifiable curve $\gamma\subset\Omega$ joining $z_{1}$
to $z_{2}$ such that
\beqq\nonumber|u(z_{1})-u(z_{2})|&\leq&\int_{\gamma}(| u_{w}(\zeta)|+| u_{\overline{w}}(\zeta)|)ds(\zeta)\\
\nonumber
&\leq&\,C_{1}\int_{\gamma}\frac{\psi(d_{\mathscr{E}}(\zeta))}{d_{\mathscr{E}}(\zeta)}\,ds(\zeta)\\
\nonumber &\leq& C\psi(|z_1-z_2|).\eeqq

We come to prove ($\mathscr{C}_{1}$)$\Leftrightarrow$($\mathscr{C}_{3}$). Since ($\mathscr{C}_{1}$)$\Rightarrow$($\mathscr{C}_{3}$) is obvious, we only need to prove ($\mathscr{C}_{3}$)$\Rightarrow$($\mathscr{C}_{1}$). 
Without loss of generality, we assume that
$\sup_{z\in\mathscr{E}}|u(z)|>0$. Let $z\in\mathscr{E}$ be fixed.
Without loss of generality, let $\zeta_{1}\in\partial\mathscr{E}$
such that $d_{\mathscr{E}}(z)=|z-\zeta_{1}|$. Since $\psi(t)/t$ is
non-increasing for $t>0$, we see that
\be\label{Ma-1}\frac{\psi(2t)}{2t}\leq\frac{\psi(t)}{t}. \ee It
follows from (\ref{Ma-1}) that, for all
$w\in\mathbb{B}^{n}(z,d_{\mathscr{E}}(z))$, there is a positive
constant $C$ such that

\beqq\label{Ma-0}
|u(w)|-|u(z)|&\leq&\big||u(z)|-|u(\zeta_{1})|\big|+\big||u(\zeta_{1})|-|u(w)|\big|\\
\nonumber
&\leq&C\psi(2d_{\mathscr{E}}(z))+C\psi(d_{\mathscr{E}}(z))\\
&\leq&3C\psi(d_{\mathscr{E}}(z)). \eeqq Consequently,
\be\label{eq-6ct}
M_{z}-|u(z)|=\sup_{w\in\mathbb{B}^{n}(z,d_{\mathscr{E}}(z))}\left(|u(w)|-|u(z)|\right)\leq3C\psi(d_{\mathscr{E}}(z)),
\ee where $M_{z}:=\sup\{|u(\xi)|:~|\xi-z|<d_{\mathscr{E}}(z)\}$. By
(\ref{eq-6ct}) and the identity Theorem, we have  $0<M_{z}<\infty$.
Let
$$\mathscr{U}(w)=\frac{u(z+d_{\mathscr{E}}(z)w)}{M_{z}},~w\in\mathbb{B}^{n}.$$

By Theorem M, we have \beqq
\frac{d_{\mathscr{E}}(z)}{M_{z}}|\nabla\,u(z)|&=&|\nabla\mathscr{U}(0)|\leq
\frac{4}{\pi}\left(1-|\mathscr{U}(0)|^{2}\right)\leq\frac{8}{\pi}\left(1-|\mathscr{U}(0)|\right)\\
&=& \frac{8}{\pi}\left(1-\frac{|u(z)|}{M_{z}}\right). \eeqq
Consequently, we have

\[
d_{\mathscr{E}}(z)|\nabla\,u(z)|\leq\frac{8}{\pi}(M_{z}-|u(z)|),
\]
which, together with (\ref{eq-6ct}), yields that \be\label{eq-7ct}
|\nabla\,u(z)|\leq\frac{24C}{\pi}\frac{\psi(d_{\mathscr{E}}(z))}{d_{\mathscr{E}}(z)}.
\ee Since ($\mathscr{C}_{2}$) implies ($\mathscr{C}_{1}$), by
(\ref{eq-7ct}), we see that $u\in\mathscr{L}_{\psi}(\mathscr{E})$.

Next, we prove
($\mathscr{C}_{1}$)$\Leftrightarrow$($\mathscr{C}_{4}$). We first
show ($\mathscr{C}_{1}$)$\Rightarrow$($\mathscr{C}_{4}$). Let  $z
\in\mathscr{E}$ be a fixed point. For
$w\in\mathbb{B}^{n}(z,d_{\mathscr{E}}(z)/2)$, it follows from
(\ref{jj-1}) that \be\label{eq-8ct} |\nabla\,u(w)|=| u_{w}(w)|+|
u_{\overline{w}}(w)| \leq\,\frac{3C_{1}}{2}\frac{
\psi(d_{\mathscr{E}}(z))}{d_{\mathscr{E}}(z)}, \ee where $C_{1}$ is
the same as in (\ref{eq-2t}). Raise both sides of (\ref{eq-8ct}) to
the $p$-th power and integrate over
$\mathbb{B}^{n}(z,d_{\mathscr{E}}(z)/2)$. This gives
 $$I_{u,p}(z)\leq\,\frac{3C_{1}}{2}\frac{\psi(d_{\mathscr{E}}(z))}{d_{\mathscr{E}}(z)}.$$

Now we begin to prove
($\mathscr{C}_{4}$)$\Rightarrow$($\mathscr{C}_{1}$). Since
$p\geq\frac{2n-2}{2n-1}$ for $n\geq 2$ and $p>0$ for $n=1$, for any
given point $z\in\mathscr{E}$, it follows from \cite[Theorem  A]{SW}
that  $|\nabla\,u|^{p}$ is subharmonic
in $\overline{\mathbb{B}^{n}(z,d_{\mathscr{E}}(z)/2)}$. 
Then

\beqq
|\nabla\,u(z)|^{p}\leq\frac{1}{|\mathbb{B}^{n}(z,d_{\mathscr{E}}(z)/2)|}\int_{\mathbb{B}^{n}(z,d_{\mathscr{E}}(z)/2)}|\nabla\,u(w)|^{p}d\mu(w),
\eeqq which implies that there is a positive constant $C$ such that

\be\label{eq-10ct}
|\nabla\,u(z)|\leq\,I_{u,p}(z)\leq\,C\frac{\psi(d_{\mathscr{E}}(z))}{d_{\mathscr{E}}(z)}.
\ee Since ($\mathscr{C}_{2}$) implies ($\mathscr{C}_{1}$), by
(\ref{eq-10ct}), we see that  ($\mathscr{C}_{4}$) implies
($\mathscr{C}_{1}$).


By using similar reasoning as in the proof of
``($\mathscr{C}_{1}$)$\Leftrightarrow$($\mathscr{C}_{2}$)$\Leftrightarrow$($\mathscr{C}_{3}$)$
\Leftrightarrow$($\mathscr{C}_{4}$)", we conclude that
($\mathscr{C}_{5}$)$\Leftrightarrow$($\mathscr{C}_{6}$)$\Leftrightarrow$($\mathscr{C}_{7}$)$\Leftrightarrow$($\mathscr{C}_{8}$).
 Since $|\nabla\,u|=|\nabla\,v|$, we see that
 ($\mathscr{C}_{2}$) is equivalent  to ($\mathscr{C}_{6}$).
 Therefore,

\be\label{pp-1}(\mathscr{C}_{1})\Leftrightarrow(\mathscr{C}_{2})\Leftrightarrow(\mathscr{C}_{3})\Leftrightarrow(\mathscr{C}_{4})\Leftrightarrow(\mathscr{C}_{5})
\Leftrightarrow(\mathscr{C}_{6})\Leftrightarrow(\mathscr{C}_{7})\Leftrightarrow(\mathscr{C}_{8}).\ee

At last, we show that
\be\label{pp-2}(\mathscr{C}_{1})\Leftrightarrow(\mathscr{C}_{9}).\ee
 Since $(\mathscr{C}_{9})\Rightarrow(\mathscr{C}_{1})$ is obvious,
we only need to prove
($\mathscr{C}_{1}$)$\Rightarrow$($\mathscr{C}_{9}$). Let
$u\in\mathscr{L}_{\psi}(\mathscr{E})$. Then
$v\in\mathscr{L}_{\psi}(\mathscr{E})$, which implies that there is a
positive constant $C$ such that \beqq
|f(z_{1})-f(z_{2})|&\leq&|u(z_{1})-u(z_{2})|+|v(z_{1})-v(z_{2})|\\
&\leq&\,C\psi(|z_{1}-z_{2}|)+C\psi(|z_{1}-z_{2}|)\\
&=&2C\psi(|z_{1}-z_{2}|) \eeqq for  $z_{1},~z_{2}\in\mathscr{E}$.
Consequently, ($\mathscr{C}_{9}$) holds. Combining (\ref{pp-1}),
(\ref{pp-2})  and Theorem E gives

\beqq
(\mathscr{C}_{1})&\Leftrightarrow&(\mathscr{C}_{2})\Leftrightarrow(\mathscr{C}_{3})\Leftrightarrow(\mathscr{C}_{4})\Leftrightarrow(\mathscr{C}_{5})
\Leftrightarrow(\mathscr{C}_{6})\Leftrightarrow(\mathscr{C}_{7})
\Leftrightarrow(\mathscr{C}_{8})\Leftrightarrow(\mathscr{C}_{9})\\
&\Leftrightarrow&(\mathscr{C}_{10})\Leftrightarrow(\mathscr{C}_{11}).
\eeqq

 The proof of this theorem is finished.
\qed

\subsection{The proof of Theorem \ref{thm-4}}
Let's prove ($\mathscr{D}_{1}$)$\Leftrightarrow$($\mathscr{D}_{2}$).
We first prove $(\mathscr{D}_{1})\Rightarrow(\mathscr{D}_{2})$. For
any fixed $z=(z_{1},\ldots, z_{n}) \in\Omega_{1}$, let
$r=2d_{\Omega_{1}}(z)/3$.  It follows from (\ref{Poisson-1t}) and
(\ref{eq-4ct}) that, for $w\in\mathbb{B}^{n}(z,3r/4)$,
\be\label{jjj-1} | u_{w}(w)|
 \leq\frac{C(n)}{r}\int_{\partial \mathbb{B}^n}|(u(z+r\xi)-u(z))|d\sigma(\xi),
 \ee where $C(n)=115\cdot4^{2n-1}n\sqrt{n}.$ Next, we estimate $|u(w_{\xi})-u(z)|$ for $\xi\in \partial \mathbb{B}^n$, where $w_{\xi}=z+r\xi$.
By the assumption, we see that there is a positive constant $C$ such
that

\be\label{eq-py-0}
|u(w_{\xi})-u(z)|\leq\,Cd_{\psi,\Omega_{1}}(w_{\xi},~z)\leq\,C\int_{[z,w_{\xi}]}\frac{\psi(d_{\Omega_{1}}(\zeta))}{d_{\Omega_{1}}(\zeta)}\,ds(\zeta),
\ee where $[z,w_{\xi}]$ is the straight segment with endpoints $z$
and $w_{\xi}$. Since
$[z,w_{\xi}]\subset\overline{\mathbb{B}^{n}(z,r)}\subset\Omega_{1}$
and $\psi(t)/t$ is non-increasing in $(0,\infty)$, we see that

\be\label{eq-py-1}
\frac{\psi(d_{\Omega_{1}}(\zeta))}{d_{\Omega_{1}}(\zeta)}\leq\frac{\psi(d_{\mathbb{B}^{n}(z,r)}(\zeta))}{d_{\mathbb{B}^{n}(z,r)}(\zeta)}
\ee for $\zeta\in[z,w_{\xi}]$. Note that
\be\label{eq-py-2}d_{\mathbb{B}^{n}(z,r)}(\zeta)=r-|\zeta-z|.\ee
Combining (\ref{eq2x}), (\ref{eq-py-0}), (\ref{eq-py-1}) and
(\ref{eq-py-2}) gives that there is a positive constant $C$ such
that

\beq\label{eq-py-3}
|u(w_{\xi})-u(z)|&\leq&\,C\int_{[z,w_{\xi}]}\frac{\psi(r-|\zeta-z|)}{r-|\zeta-z|}\,ds(\zeta)=\,C\int_{0}^{r}\frac{\psi(t)}{t}dt
\\ \nonumber &\leq&C\psi(r),
\eeq where the last inequality follows from the assumptions that
$\psi$ is a fast majorant and $\Omega_1$ is bounded. Since $\psi$ is
a continuous increasing function, by (\ref{jjj-1}) and
(\ref{eq-py-3}),  we conclude that

\beq\label{eq-py-4} | u_{w}(w)|&\leq&\frac{C\cdot
C(n)}{r}\int_{\partial \mathbb{B}^n}\psi(r)d\sigma(\xi)=C\cdot
C(n)\frac{\psi(r)}{r}\\ \nonumber &=& C\cdot
C(n)\frac{\psi\left(\frac{2}{3}d_{\Omega_{1}}(z)\right)}{\frac{2}{3}d_{\Omega_{1}}(z)}\\
\nonumber &\leq&\frac{3}{2}C\cdot
C(n)\frac{\psi\left(d_{\Omega_{1}}(z)\right)}{d_{\Omega_{1}}(z)}.
 \eeq
Taking $w=z$ in (\ref{eq-py-4}), we obtain \beqq |\nabla\,u(z)|=|
u_{w}(z)|+| u_{\overline{w}}(z)|=2| u_{w}(z)|\leq 3C\cdot
C(n)\frac{\psi\left(d_{\Omega_{1}}(z)\right)}{d_{\Omega_{1}}(z)}.
\eeqq

Now, we  prove $(\mathscr{D}_{2})\Rightarrow(\mathscr{D}_{1})$. By
the assumption, we see that, for $z_{1},~z_{2}\in\Omega_{1}$,

\beqq
|u(z_{1})-u(z_{2})|&\leq&\inf\int_{\gamma}|\nabla\,u(\zeta)|ds(\zeta)\leq
C\cdot \inf\int_{\gamma}\frac{\psi(d_{\Omega_{1}}(z))}{d_{\Omega_{1}}(z)}\,ds(z)\\
&=& C\cdot d_{\psi,\Omega_{1}}(z_{1},~z_{2}), \eeqq where the
infimum is taken over all rectifiable curves
$\gamma\subset\Omega_{1}$ joining $z_{1}$ to $z_{2}$. Consequently,
$u\in\mathscr{L}_{\psi,{\rm\,int}}(\Omega_{1})$.

Next, we  prove
($\mathscr{D}_{1}$)$\Leftrightarrow$($\mathscr{D}_{3}$). Since
($\mathscr{D}_{1}$)$\Rightarrow$($\mathscr{D}_{3}$) is obvious, we
only need to prove
($\mathscr{D}_{3}$)$\Rightarrow$($\mathscr{D}_{1}$). Without loss of
generality, we assume that $\sup_{z\in\Omega_{1}}|u(z)|>0$. Let
$z\in\Omega_{1}$ be fixed. For
$w\in\mathbb{B}^{n}(z,d_{\Omega_{1}}(z))$, there is a positive
constant $C$ such that

\be\label{eq-chh-11}
|u(w)|-|u(z)|\leq\,Cd_{\psi,\Omega_{1}}(w,~z)\leq\,C\int_{[w,z]}\frac{\psi(d_{\Omega_{1}}(\zeta))}{d_{\Omega_{1}}(\zeta)}\,ds(\zeta),
\ee where $[w,z]$ denotes the straight segment with endpoints $w$
and $z$. We observe that if $\zeta\in[w,z]$, then one has
$$[w,z]\subset\mathbb{B}^{n}(z,d_{\Omega_{1}}(z))\subset\Omega_{1}$$
and therefore

$$d_{\Omega_{1}}(\zeta)\geq\,d_{\mathbb{B}^{n}(z,d_{\Omega_{1}}(z))}(\zeta).$$
This gives

\be\label{eq-chh-12}\frac{\psi(d_{\Omega_{1}}(\zeta))}{d_{\Omega_{1}}(\zeta)}\leq\frac{\psi(d_{\mathbb{B}^{n}(z,d_{\Omega_{1}}(z))}(\zeta))}{d_{\mathbb{B}^{n}(z,d_{\Omega_{1}}(z))}(\zeta)}.\ee
Since $\psi$ is fast and $\Omega_1$ is bounded, by (\ref{eq-chh-11})
and (\ref{eq-chh-12}), we see  that there is a positive constant $C$
such that

\beqq
|u(w)|-|u(z)|&\leq&\,C\int_{[w,z]}\frac{\psi(d_{\mathbb{B}^{n}(z,d_{\Omega_{1}}(z))}(\zeta))}{d_{\mathbb{B}^{n}(z,d_{\Omega_{1}}(z))}(\zeta)}\,ds(\zeta)\\
&=&C\int_{[w,z]}\frac{\psi(d_{\Omega_{1}}(z)-|\zeta-z|)}{d_{\Omega_{1}}(z)-|\zeta-z|}\,ds(\zeta)\\
&\leq &C\int_{0}^{d_{\Omega_{1}}(z)}\frac{\psi(t)}{t}dt\\
&\leq&\,C\psi(d_{\Omega_{1}}(z)). \eeqq From this, we obtain that
\be\label{eq-chh-13} M_{z}-|u(z)|\leq\,C\psi(d_{\Omega_{1}}(z)),\ee
where $M_{z}:=\sup\{|u(\xi)|:~|\xi-z|<d_{\Omega_{1}}(z)\}$. By
(\ref{eq-chh-13}) and the identity principle, we have
$0<M_z<\infty$. Let
$$U(\eta)=\frac{u(z+d_{\Omega_{1}}(z)\eta)}{M_{z}},~\eta\in\mathbb{B}^{n},$$

From the proof of Theorem \ref{thm-3}, we have \be\label{eq-chh10}
d_{\Omega_{1}}(z)|\nabla\,u(z)|\leq\frac{8}{\pi}(M_{z}-|u(z)|). \ee
Combining (\ref{eq-chh-13}) and (\ref{eq-chh10}) yields that there
is a positive constant $C$ such that

\be\label{eq-100dt}
|\nabla\,u(z)|\leq\,C\frac{\psi(d_{\Omega_{1}}(z))}{d_{\Omega_{1}}(z)}.
\ee Since $(\mathscr{D}_{2})$ implies $(\mathscr{D}_{1})$, by
(\ref{eq-100dt}), we see that  $(\mathscr{D}_{3})$ implies
$(\mathscr{D}_{1})$.

We prove $(\mathscr{D}_{1})\Leftrightarrow(\mathscr{D}_{4})$. We
first show $(\mathscr{D}_{1})\Rightarrow(\mathscr{D}_{4})$. Let  $z
\in\Omega_{1}$ be a fixed point. For
$w\in\mathbb{B}^{n}(z,d_{\Omega_{1}}(z)/2)$, it follows from
(\ref{eq-py-4}) that \be\label{eq-8dt} |\nabla\,u(w)|=| u_{w}(w)|+|
u_{\overline{w}}(w)| \leq\,\frac{3C_{1}}{2}\frac{
\psi(d_{\Omega_{1}}(z))}{d_{\Omega_{1}}(z)}, \ee where $C_{1}=C\cdot
C(n)$ is the same as in (\ref{eq-py-4}). Raise both sides of
(\ref{eq-8dt}) to the $p$-th power and integrate over
$\mathbb{B}^{n}(z,d_{\Omega_{1}}(z)/2)$. This gives
 $$I_{u,p}(z)\leq\,\frac{3C_{1}}{2}\frac{\psi(d_{\Omega_{1}}(z))}{d_{\Omega_{1}}(z)}.$$

Now we begin to prove
($\mathscr{D}_{4}$)$\Rightarrow$($\mathscr{D}_{1}$). For any given
point $z\in\Omega_{1}$, it follows from \cite[Theorem  A]{SW} that
$|\nabla\,u|^{p}$ is subharmonic in
$\overline{\mathbb{B}^{n}(z,d_{\Omega_{1}}(z)/2)}$. Then

\beqq
|\nabla\,u(z)|^{p}\leq\frac{1}{|\mathbb{B}^{n}(z,d_{\Omega_{1}}(z)/2)|}\int_{\mathbb{B}^{n}(z,d_{\Omega_{1}}(z)/2)}|\nabla\,u(w)|^{p}d\mu(w),
\eeqq which implies that there is a positive constant $C$ such that

\be\label{eq-10dt}
|\nabla\,u(z)|\leq\,I_{u,p}(z)\leq\,C\frac{\psi(d_{\Omega_{1}}(z))}{d_{\Omega_{1}}(z)}.
\ee Since $(\mathscr{D}_{2})$ implies $(\mathscr{D}_{1})$, by
(\ref{eq-10dt}), we see that  $(\mathscr{D}_{4})$ implies
$(\mathscr{D}_{1})$.

Since the remaining proof of this theorem is similar to the proof of
Theorem \ref{thm-3}, we omit it here. The proof of this theorem is
finished. \qed

\begin{ThmN}{\rm (\cite[Theorem 1.8]{KV})}\label{KV-2012}
Let $u:~\mathbf{B}^{2}\rightarrow(-1,1)$ be a harmonic function.
Then
$$|\nabla\,u(x)|\leq\frac{4}{\pi}\frac{1-|u(x)|^{2}}{1-|x|^2},
\quad x\in \mathbf{B}^{2}.$$ In addition, this inequality is sharp.
\end{ThmN}

\begin{ThmO}{\rm (\cite[Theorem 2.5]{KPM})}\label{KPM-2021}
Let $n\geq3$ and $u:~\mathbf{B}^{n}\rightarrow(-1,1)$ be a harmonic
function. Then
$$|\nabla\,u(x)|\leq\frac{n}{2}\frac{1-|u(x)|^{2}}{1-|x|},\quad
x\in \mathbf{B}^{n}.$$ In addition, this inequality is strict for
$n\geq4$.
\end{ThmO}

\subsection{The proof of Theorem \ref{Har-1}}
We first prove ($\mathscr{F}_{1}$)$\Rightarrow$($\mathscr{F}_{2}$).
For any fixed $x\in\mathbf{B}^{n}$, let $r=2d_{\Omega_{2}}(x)/3$.
For $y\in\overline{\mathbf{B}^{n}(x,r)}$, it follows from the
Poisson integral formula that
$$u(y)=\int_{\mathbf{S}^{n-1}}\mathbf{P}_{r}(y-x,\zeta)u(r\zeta+x)d\sigma(\zeta),$$
where
$\mathbf{P}_{r}(y-x,\zeta)=r^{n-2}\left(r^{2}-|y-x|^{2}\right)/|y-x-r\zeta|^{n}$
is the Poisson kernel for $\mathbf{B}^{n}(x,r)$. By subtracting
$u(x)$ from both sides and differentiating them with respect to
$y\in\mathbf{B}^{n}(x,r)$, we have
\be\label{Ha-ch-1}\nabla\,u(y)=\int_{\mathbf{S}^{n-1}}\nabla\mathbf{P}_{r}(y-x,\zeta)(u(r\zeta+x)-u(x))d\sigma(\zeta).\ee
If $y\in\mathbf{B}^{n}(x,3r/4)$, then there is a positive constant
$C$ depending only on $n$ such that \be\label{Ha-ch-2}
|\nabla\mathbf{P}_{r}(y-x,\zeta)|\leq\frac{C}{r}~\mbox{for}~\zeta\in\mathbf{S}^{n-1}.
\ee From the assumption, we see that there is a positive constant
$C$ such that

\be\label{Ha-ch-3} |u(r\zeta+x)-u(x)|\leq\,C\psi(r). \ee Combining
(\ref{Ha-ch-1}), (\ref{Ha-ch-2}) and (\ref{Ha-ch-3}) yields that
there is a positive constant $C$ such that

\be\label{Ha-ch-4}
|\nabla\,u(y)|\leq\,C\frac{\psi(r)}{r}=\frac{3C}{2}\frac{\psi\left(\frac{2}{3}d_{\Omega_{2}}(x)\right)}{d_{\Omega_{2}}(x)}
\leq\frac{3C}{2}\frac{\psi(d_{\Omega_{2}}(x))}{d_{\Omega_{2}}(x)}.
\ee Taking $y=x$ in (\ref{Ha-ch-4}), we have
$$|\nabla\,u(x)|\leq\frac{3C}{2}\frac{\psi(d_{\Omega_{2}}(x))}{d_{\Omega_{2}}(x)}.$$

The proof of ($\mathscr{F}_{2}$)$\Rightarrow$($\mathscr{F}_{1}$) is
similar to the proof of Theorem \ref{thm-3}. Also, if we replace
Theorem M by Theorems  N and O in the proof of Theorem \ref{thm-3},
then we can prove
($\mathscr{F}_{1}$)$\Leftrightarrow$($\mathscr{F}_{3}$). The proof
of ($\mathscr{F}_{1}$)$\Leftrightarrow$($\mathscr{F}_{4}$) is
similar to the proof
($\mathscr{F}_{1}$)$\Leftrightarrow$($\mathscr{F}_{3}$).

At last, we show
($\mathscr{F}_{1}$)$\Leftrightarrow$($\mathscr{F}_{5}$). We first
prove ($\mathscr{F}_{1}$)$\Rightarrow$($\mathscr{F}_{5}$). Let
$x\in\Omega_{2}$ be a fixed point. For $y\in\mathbf{B}^{n}(x,
d_{\Omega_{2}}(x)/2)$, it follows from (\ref{Ha-ch-4}) that there is
a positive constant $C$ such that

\be\label{Ha-ch-6} |\nabla\,u(y)|
\leq\frac{3C}{2}\frac{\psi(d_{\Omega_{2}}(x))}{d_{\Omega_{2}}(x)}.
\ee Raise both sides of (\ref{Ha-ch-6}) to the $p$-th power and
integrate over $\mathbf{B}^{n}(x, d_{\Omega_{2}}(x)/2)$. This gives
$$\tilde{I}_{u,p}(x)\leq\,\frac{3C}{2}\frac{\psi(d_{\Omega_{2}}(x))}{d_{\Omega_{2}}(x)}.$$

Next we begin to prove
($\mathscr{F}_{5}$)$\Rightarrow$($\mathscr{F}_{1}$). Since
$p\geq\frac{n-2}{n-1}$ for $n> 2$ and $p>0$ for $n=2$, for any given
point $x\in\Omega_{2}$, it follows from \cite[Theorem  A]{SW} that
$|\nabla\,u|^{p}$ is subharmonic
in $\overline{\mathbf{B}^{n}(x, d_{\Omega_{2}}(x)/2)}$. 
Then

\beqq |\nabla\,u(x)|^{p}\leq\frac{1}{|\mathbf{B}^{n}(x,
d_{\Omega_{2}}(x)/2)|}\int_{\mathbf{B}^{n}(x,
d_{\Omega_{2}}(x)/2)}|\nabla\,u(x)|^{p}d\mu(x), \eeqq which implies
that there is a positive constant $C$ such that

\be\label{Ha-ch-8}
|\nabla\,u(x)|\leq\,\tilde{I}_{u,p}(x)\leq\,C\frac{\psi(d_{\Omega_{2}}(x))}{d_{\Omega_{2}}(x)}.
\ee Since ($\mathscr{F}_{2}$) implies ($\mathscr{F}_{1}$), by
(\ref{Ha-ch-8}), we see that  ($\mathscr{F}_{5}$) implies
($\mathscr{F}_{1}$). The proof of this theorem is finished. \qed



\section{Acknowledgments}
 The research of the first author was partly supported by the National Science
Foundation of China (grant no. 12071116), the Hunan Provincial
Natural Science Foundation of China (No. 2022JJ10001), the Key
Projects of Hunan Provincial Department of Education (grant no.
21A0429),
 the Double First-Class University Project of Hunan Province
(Xiangjiaotong [2018]469),  the Science and Technology Plan Project
of Hunan Province (2016TP1020),  and the Discipline Special Research
Projects of Hengyang Normal University (XKZX21002); The research of
the second author was partly supported by JSPS KAKENHI Grant Number
JP22K03363.

\end{document}